\documentclass[11pt]{article}

\usepackage[a4paper,margin=0.8in]{geometry}
\usepackage{amsmath,amssymb,amsthm,mathtools}
\usepackage{microtype}
\usepackage{booktabs}
\usepackage{array}
\usepackage[hidelinks]{hyperref}
\usepackage{enumitem}
\usepackage{graphicx}
\usepackage{tikz}
\usetikzlibrary{positioning,arrows.meta}

\numberwithin{equation}{section}
\newtheorem{theorem}{Theorem}[section]
\newtheorem{lemma}[theorem]{Lemma}
\newtheorem{proposition}[theorem]{Proposition}
\newtheorem{corollary}[theorem]{Corollary}
\theoremstyle{definition}

\newcommand{\crn}{\operatorname{cr}}
\newcommand{\ZZ}{\mathrm{Z}}

\newcommand{\ceil}[1]{\left\lceil #1\right\rceil}
\newcommand{\floor}[1]{\left\lfloor #1\right\rfloor}

\title{A Proof of Albertson's Conjecture}  
\author{
Ankan Sadhu\\[1.5mm]
{\small Department of Computer Science and Engineering}\\[-1mm]
{\small Government College of Engineering and Ceramic Technology}\\[-1mm]
{\small Kolkata, India}\\[-1mm]
{\small\texttt{ankan-sadhu-br24-3003@gcect.ac.in}}
}
\date{}
\usepackage[section]{placeins}
\usepackage{needspace}
\allowdisplaybreaks[0]

\usetikzlibrary{calc,decorations.pathreplacing}
\definecolor{proofblue}{RGB}{29,76,114}
\definecolor{prooforange}{RGB}{155,75,20}
\definecolor{proofteal}{RGB}{25,104,96}
\tikzset{
  proofbox/.style={draw=black!55,rounded corners=2pt,fill=black!2,
    align=center,inner sep=7pt,font=\fontsize{10}{12}\selectfont},
  proofarrow/.style={-{Stealth[length=2.2mm]},semithick},
  proofvertex/.style={circle,draw=black,fill=white,inner sep=0pt,minimum size=5pt},
  proofbranch/.style={circle,draw=proofblue,fill=proofblue,inner sep=0pt,minimum size=6pt},
  proofoutside/.style={rectangle,draw=prooforange,fill=white,inner sep=0pt,minimum size=6pt},
  proofpanel/.style={font=\small\bfseries,anchor=west},
  proofnote/.style={font=\fontsize{10}{12}\selectfont,align=center},
  proofroute/.style={draw=proofblue,line width=1.2pt},
  proofreserve/.style={draw=prooforange,densely dashed,line width=.8pt}
}

\begin{document}
\maketitle
\begin{abstract}
Albertson's conjecture states that every graph of chromatic number $r$ has crossing number at least that of the complete graph $K_r$ on $r$ vertices. We prove the conjecture for all $r$. The proof combines clique-sensitive edge bounds for critical graphs, sampling inside and outside a maximum clique, and an immersion construction. In the construction, crossings in a reserved bipartite subgraph compensate for the crossings introduced when the immersion is converted into a drawing. The remaining numerical inequalities are verified by exact rational arithmetic using a supplied certificate and a standalone verification program.
\end{abstract}


\section{Introduction}

For a finite simple graph $G$, let $\chi(G)$ denote its chromatic number and let $\crn(G)$ denote the minimum number of edge crossings in a drawing of $G$ in the plane. In 2007, Albertson conjectured that the complete graph $K_r$ minimizes crossing number among all graphs of chromatic number $r$~\cite[Conjecture~2]{ACF}.

The conjecture is immediate for $r\le4$. Its case $r=5$ is equivalent to the Four Colour Theorem, since $\crn(K_5)=1$. For larger $r$, a graph of chromatic number $r$ need not contain $K_r$, so a proof must obtain the crossing lower bound without assuming such a clique is present. We prove the conjecture for all $r$.

\begin{theorem}[Albertson's conjecture]\label{thm:main}
For every finite simple graph $G$ and positive integer $r$, if $\chi(G)\ge r$, then
\[
\crn(G)\ge\crn(K_r).
\]
\end{theorem}

\paragraph{Related work.}
Progress on the conjecture has followed two complementary directions. The first combines lower bounds on the number of edges in critical graphs with crossing inequalities. Oporowski and Zhao proved the case $r=6$~\cite[Corollary~2.4]{OporowskiZhao}. Albertson, Cranston and Fox established the conjecture for $r\le12$ in 2009~\cite[Section~4]{ACF}, Bar\'at and T\'oth extended this to $r\le16$ in 2010~\cite[Theorem~2]{BaratToth}, and Ackerman obtained $r\le18$~\cite[Section~3.1]{Ackerman}. Critical-graph edge estimates, crossing inequalities, and averaging over sampled subgraphs are central to this line of work. These estimates also exclude critical counterexamples of sufficiently large order, leaving an intermediate range in which sharper bounds are required.

The second direction seeks a representation of $K_r$ inside the graph. A subdivision of $K_r$ immediately gives the desired crossing inequality. Bar\'at and T\'oth proved that an $r$-critical graph on at most $r+4$ vertices contains such a subdivision~\cite[Corollary~11]{BaratToth}; Luiz and Richter extended this to $r+5$~\cite[Theorem~1]{LuizRichter}. Fox, Pach and Suk replaced internally vertex-disjoint paths by edge-disjoint paths, allowing a weak immersion~\cite[Section~2, proof of Theorem~2]{FoxPachSuk}. Converting such an immersion into a drawing of $K_r$ may introduce crossings at shared vertices. Their method bounds this increase by crossings forced by unused edges. Their published result proves Albertson's conjecture for $n\le(1.4+\delta)r$, for some $\delta>0$ and sufficiently large $r$~\cite[Theorem~4]{FoxPachSuk}; an expanded preprint reaches $n\le(1.64-\varepsilon)r$ for each fixed $\varepsilon>0$ and sufficiently large $r$~\cite[Theorem~1.4]{FoxPachSukExpanded}.

Recent preprints report further finite-range progress. Cranston proves the conjecture for $r\le24$ and restricts the orders of possible counterexamples~\cite[Theorems~2--4]{Cranston}, while Cao and Mehat treat $r\le29$~\cite[Theorem~1.1]{CaoMehat}. These preprints and the expanded preprint of Fox, Pach and Suk describe related progress. The proof uses Ackerman's published result for $r\le18$ and the arguments below for all $r\ge19$.

\paragraph{Outline of the proof.}
A graph is \emph{$r$-critical} if its chromatic number is $r$ and every proper subgraph is $(r-1)$-colourable. Every $r$-chromatic graph contains such a subgraph, and taking a subgraph cannot increase crossing number. We therefore work with an $r$-critical graph of order $n$. For $n$ close to $r$, we construct an immersion with an explicit bound on the crossings introduced by smoothing. For larger orders, we combine an edge lower bound depending on clique number with a crossing lower bound for a maximum clique. Sampling separately inside and outside that clique incorporates both estimates into a single inequality.

The proof has three ranges. An immersion construction, supplemented by a residual sampling calculation, treats $n<5r/4$. Maximum-clique sampling treats $5r/4\le n\le3r$, and minimum degree together with ordinary sampling treats $n\ge3r$. The estimates in all three ranges apply to every $r\ge19$.

\paragraph{The near range.}
Write $n=r+c$. Gallai's decomposition~\cite[Satz~(E\,2.1)]{Gallai} expresses the critical graph as a join of smaller critical graphs. When $c$ is small, many factors are single vertices, each adjacent to every vertex outside its factor. The singleton factors serve as internal vertices of routes between nonadjacent branch vertices. Following the immersion framework of Fox, Pach and Suk, we use injective choices of entrance edges and Shannon's multigraph edge-colouring theorem~\cite[main theorem]{Shannon} to make the routes edge-disjoint.

Lemma~\ref{lem:reserved} constructs the routes while leaving the edges of a complete bipartite graph $K_{c,b}$ unused. Lemma~\ref{lem:smoothing} bounds the crossings introduced by smoothing in terms of the number of nontrivial paths and the number of strands remaining at each local modification. A lower bound derived from Kleitman's bipartite crossing formula~\cite[Abstract]{Kleitman} then bounds this increase. The common reserve and the explicit finite smoothing estimate are the refinements of the immersion method used here.

\paragraph{Clique-sensitive edge bounds.}
The density of critical graphs under clique restrictions has been studied by, among others, Krivelevich~\cite[Theorem~1]{Krivelevich} and Kelly and Postle~\cite[Theorem~1.1]{KellyPostle}. We need bounds that retain explicit dependence on the order and clique number. To obtain them, we repeatedly delete an independent triple and retain a critical subgraph of chromatic number one less. Independent-set deletion with retention of critical subgraphs is an established method; see Fox, Tidor and Zhang~\cite[proof of Theorem~2.2]{FoxTidorZhang}. Here two elementary degree counts control the edge loss, including the vertices discarded when the next critical subgraph is chosen.

The descent ends at a graph with independence number at most two. Gallai's theorem and colouring after vertex deletion determine the orders of its critical join factors, while their triangle-free complements bound the missing edges. Combining this terminal estimate with the accumulated deletion losses gives the explicit polynomial bounds of Theorem~\ref{thm:clique-edge}. Two factored identities give the final comparison.

\paragraph{Maximum-clique sampling.}
Ordinary induced-subgraph averaging records how often edges and crossings survive a sample; see, for example, Bar\'at and T\'oth~\cite[Section~5]{BaratToth}. Our sample instead fixes the number of vertices chosen inside and outside a maximum clique. Lemma~\ref{lem:max-clique-test} combines the edge estimates with a correction for crossings already forced inside that actual clique. The resulting inequality incorporates both the stronger edge estimates for small clique number and the crossings forced by a large clique. Log-concavity reduces the potentially controlling crossing types to three, and cancellation gives polynomials of degree at most six. For $n\ge3r$, the single ordinary sample size $\lceil12n/r\rceil$ is sufficient.

The numerical comparisons use the B\"ungener--Kaufmann crossing inequality~\cite[Theorem~4(b)]{BK}. Harary--Hill~\cite[Section~3, equation~(2)]{HararyHill} supplies a construction upper bound for $\crn(K_r)$; no conjectured formula for that crossing number is assumed. After the reductions, computer assistance checks positivity of explicitly displayed rational polynomials by standard Bernstein range enclosure~\cite[Section~1.1, Theorem~1(i)]{Garloff}. A 463-box certificate covers the bounded order range for every $r\ge19$. The residual near range and the large range require 25 and 150 coefficients, respectively. The supplied standard-Python verifier checks the exact coefficients, the sampling admissibility guards and edge-bound guards, and coverage of the parameter domain.

Section~\ref{sec:prelim} states the published inputs. Section~\ref{sec:compression} proves the edge estimates, Section~\ref{sec:near} gives the near-range construction, and Section~\ref{sec:middle} develops the sampling arguments. Section~\ref{sec:final} assembles the proof. Appendix~\ref{app:verification} gives the coefficient criterion and verification details.

The three order ranges are summarized in Figure~\ref{fig:roadmap}.
\begin{figure}[!htbp]
\centering
\scalebox{0.96}{%
\begin{tikzpicture}[x=1cm,y=1cm]
\node[proofbox,text width=13.6cm] (start) at (7.5,0)
  {$G$ is $r$-critical. The published cases cover $r\le18$;\\
   the argument below treats $r\ge19$, with $n=|G|$.};
\node[proofbox,text width=4.05cm,minimum height=1.15cm] (near) at (2.35,-2)
  {$n<5r/4$\\[2pt]\textbf{Near range}};
\node[proofbox,text width=4.05cm,minimum height=1.15cm] (middle) at (7.5,-2)
  {$5r/4\le n\le3r$\\[2pt]\textbf{Bounded order ratio}};
\node[proofbox,text width=4.05cm,minimum height=1.15cm] (large) at (12.65,-2)
  {$n\ge3r$\\[2pt]\textbf{Large order}};
\foreach \n in {near,middle,large}{\draw[proofarrow] (start.south) -- ++(0,-.3) -| (\n.north);}
\node[proofbox,text width=4.05cm,minimum height=3.5cm] (nwork) at (2.35,-5)
  {If $c=n-r\le5$: subdivision.\\[3pt]
   Otherwise: Gallai factors\\$\downarrow$\\routes and reserve $K_{c,b}$\\$\downarrow$\\reserved crossings\\versus smoothing cost};
\node[proofbox,text width=4.05cm,minimum height=3.5cm] (mwork) at (7.5,-5)
  {Small clique: edge bound\\Large clique: crossings\\$\downarrow$\\maximum-clique sampling\\$\downarrow$\\one combined inequality};
\node[proofbox,text width=4.05cm,minimum height=3.5cm] (lwork) at (12.65,-5)
  {Minimum degree $r-1$\\$\downarrow$\\ordinary sample of size\\$\lceil12n/r\rceil$\\$\downarrow$\\one polynomial inequality};
\foreach \a/\b in {near/nwork,middle/mwork,large/lwork}{\draw[proofarrow] (\a) -- (\b);}
\node[proofbox,text width=4.05cm,minimum height=1.7cm] (nend) at (2.35,-8.65)
  {Reserve estimate for\\the main near range.\\25 positive coefficients\\for the residual cases.};
\node[proofbox,text width=4.05cm,minimum height=1.7cm] (mend) at (7.5,-8.65)
  {463 boxes cover all\\$(n/r,\omega(G)/r,1/r)$.\\Positive polynomial\\on every box.};
\node[proofbox,text width=4.05cm,minimum height=1.7cm] (lend) at (12.65,-8.65)
  {One fixed parameter box.\\150 positive coefficients\\cover all $n\ge3r$.};
\foreach \a/\b in {nwork/nend,mwork/mend,lwork/lend}{\draw[proofarrow] (\a) -- (\b);}
\node[proofbox,text width=13.6cm,draw=proofblue] (end) at (7.5,-10.6)
  {Conclusion: $\crn(G)\ge\crn(K_r)$.};
\foreach \n in {nend,mend,lend}{\draw[proofarrow] (\n.south) -- ++(0,-.4) -| (end.north);}
\end{tikzpicture}
}%
\caption{Structure of the proof. The near range uses subdivisions and a reserved immersion, with a sampling calculation for the residual cases. The other two ranges use sampling inequalities. Both adjacent arguments apply at $n=3r$.}
\label{fig:roadmap}
\end{figure}

\section{Preliminaries}\label{sec:prelim}

All graphs are finite and simple unless a multigraph is explicitly introduced. We write $\alpha(G)$ and $\omega(G)$ for independence and clique number, and $G\vee H$ for the join of disjoint graphs, formed by adding every edge between their vertex sets. For a graph $H$, write $|H|=|V(H)|$ and $m(H)=|E(H)|$. Unless otherwise stated, $G$ is an $r$-critical graph with $n=|G|$ and $m=m(G)$. For a nonnegative integer $j$, write
\[
(a)_j=a(a-1)\cdots(a-j+1),\qquad (a)_0=1.
\]
For integers $a\ge j$, this is the number of ordered choices of $j$ distinct elements from an $a$-element set. These falling factorials record the survival probabilities in the sampling arguments.

A drawing is called \emph{good} if no edge crosses itself, adjacent edges do not cross, any two edges cross at most once, and no three edges cross at one interior point. A crossing-minimal drawing may always be taken good. In particular, every crossing has four distinct endpoints, and any fixed four vertices support at most three crossings.

A \emph{weak immersion} of a graph $H$ in $G$ assigns to each vertex of $H$ a distinct vertex of $G$, called a branch vertex, and to each edge of $H$ a path in $G$ between its assigned endpoints. These paths must be pairwise edge-disjoint. The paths may share internal vertices and may pass internally through another branch vertex. The crossings introduced when separating these paths are bounded in Lemma~\ref{lem:smoothing}.

Put
\[
\ZZ(r):=\frac14
\floor{\frac r2}
\floor{\frac{r-1}{2}}
\floor{\frac{r-2}{2}}
\floor{\frac{r-3}{2}}.
\]
The Harary--Hill construction~\cite[Section~3, equation~(2)]{HararyHill} gives a drawing of $K_r$ with $\ZZ(r)$ crossings, so
\[
\crn(K_r)\le \ZZ(r).
\]
Thus the numerical parts of the proof may establish the stronger inequality
\begin{equation}
\crn(G)\ge \ZZ(r),
\label{eq:targetZ}
\end{equation}
since then $\crn(G)\ge\ZZ(r)\ge\crn(K_r)$. The subdivision and reserved-immersion arguments in Section~\ref{sec:near} compare directly with $\crn(K_r)$. Only the sampling arguments use the stronger numerical target \eqref{eq:targetZ}; nowhere do we need the equality $\crn(K_r)=\ZZ(r)$.

\begin{lemma}[Counterexample consequences]\label{lem:counter-facts}
If $G$ is an $r$-critical counterexample to Albertson's conjecture, then
\begin{equation}
\delta(G)\ge r-1,
\qquad
\crn(G)\le \ZZ(r)-1.
\label{eq:counter}
\end{equation}

\end{lemma}

\begin{proof}
If a vertex $v$ had degree at most $r-2$, an $(r-1)$-colouring of $G-v$ could assign $v$ a colour absent from its neighbours. This contradicts $\chi(G)=r$, so $\delta(G)\ge r-1$.  Since $G$ is a counterexample,
\[
\crn(G)<\crn(K_r)\le\ZZ(r),
\]
and crossing numbers are integers, giving the second assertion in \eqref{eq:counter}.
\end{proof}

We use the following polynomial upper bound in the sampling arguments:
\begin{equation}
\overline Z(r)=\frac{(r-1)^2(r-3)^2}{64},\qquad
T(x)=\frac{(1-x)^2(1-3x)^2}{64}-x^4.
\label{eq:smooth-target}
\end{equation}
For odd $r$, $\ZZ(r)=\overline Z(r)$. For even $r$, expansion gives
\[
\overline Z(r)-\ZZ(r)=\frac{2(r-2)^2+1}{64}>0.
\]
Thus a counterexample has at most $\overline Z(r)-1=r^4T(1/r)$ crossings. This uses only the Harary--Hill construction upper bound.

Whenever we argue from a counterexample, we use these two consequences without further comment.  A real lower bound strictly larger than $\ZZ(r)-1$ therefore implies $\crn(G)\ge\ZZ(r)$ by integrality.

\subsection{Crossing inequalities}

We use Theorem~4(b) of B\"ungener and Kaufmann~\cite[Theorem~4(b)]{BK}.

\begin{lemma}[B\"ungener--Kaufmann]\label{lem:BK}
For every graph $H$ with $N>2$ vertices and $M$ edges,
\begin{equation}
\crn(H)\ge5M-\frac{203}{9}(N-2).
\label{eq:BK5}
\end{equation}
\end{lemma}
Since crossing numbers and $M$ are integers, this also gives
\begin{equation}
\crn(H)\ge5M-\floor{\frac{203}{9}(N-2)}.
\label{eq:BK5-int}
\end{equation}

Every sampled drawing is inherited from a fixed drawing of the original graph. Its crossing count is at least the crossing number of the sampled graph, so the crossing inequalities apply to each inherited drawing.

The following fixed-size averaging argument appears in Ackerman~\cite[Section~3.1]{Ackerman}. We state it with general coefficients for use with the crossing inequality above.

\begin{lemma}[Induced-subgraph averaging]\label{lem:sampling}
Suppose every graph $H$ on $s$ vertices satisfies
\[
\crn(H)\ge a|E(H)|-b(s-2).
\]
If $G$ has $n\ge s\ge4$ vertices and $m$ edges, then
\begin{equation}
\crn(G)\ge
 a m\frac{(n-2)(n-3)}{(s-2)(s-3)}
 -b\frac{n(n-1)(n-2)(n-3)}{s(s-1)(s-3)}.
\label{eq:sample}
\end{equation}
\end{lemma}

\begin{proof}
Take a crossing-minimal good drawing of $G$ and sum over all induced $s$-vertex subdrawings. Each edge occurs in $\binom{n-2}{s-2}$ samples and each crossing in $\binom{n-4}{s-4}$ samples. If $C(S)$ counts crossings in the inherited drawing on $S$, then
\[
\binom{n-4}{s-4}\crn(G)
=\sum_{|S|=s}C(S)
\ge a\binom{n-2}{s-2}m-b(s-2)\binom ns.
\]
The two multiplicities give
\[
\frac{\binom{n-2}{s-2}}{\binom{n-4}{s-4}}
=\frac{(n-2)(n-3)}{(s-2)(s-3)},
\qquad
\frac{(s-2)\binom ns}{\binom{n-4}{s-4}}
=\frac{n(n-1)(n-2)(n-3)}{s(s-1)(s-3)}.
\]
Dividing by $\binom{n-4}{s-4}>0$ and using these ratios gives \eqref{eq:sample}.
\end{proof}

\subsection{Critical-graph edge bounds}

Alongside the clique-sensitive bounds proved in Section~\ref{sec:compression}, we use two standard edge estimates: minimum degree and Gallai's small-order estimate~\cite[Satz~(E\,2.3)]{Gallai} (see also \cite[Lemma~4]{BaratToth}).

\begin{lemma}\label{lem:critical-edges}
Let $G$ be $r$-critical with $r\ge4$. Then
\begin{align}
m&\ge\ceil{(r-1)n/2},\label{eq:degree-edge}\\
m&\ge\ceil{\frac{(r-1)n+(n-r)(2r-n)-2}{2}}
\quad(r+2\le n\le2r-1).\label{eq:Gallai-edge}
\end{align}
\end{lemma}
The first inequality follows by summing minimum degrees. The minimum-degree argument in Lemma~\ref{lem:counter-facts} applies to every critical graph.

\subsection{Gallai decomposition}

\begin{lemma}[Gallai~{\cite[Satz~(E\,2.1)]{Gallai}}]\label{lem:gallai-join}
If $G$ is $r$-critical and $n\le2r-2$, then $\overline G$ is disconnected. More precisely,
\[
G=G_1\vee\cdots\vee G_t,
\]
where $G_i$ is $r_i$-critical, $\sum_i r_i=r$, and every non-singleton part satisfies
\[
|G_i|\ge2r_i-1.
\]
\end{lemma}

Take the parts to be the subgraphs induced by the connected components of $\overline G$. Their join is $G$, and colours used in different parts must be disjoint. Each part is critical, since replacing it by a proper subgraph of the same chromatic number would contradict criticality of $G$. Applying the disconnected-complement conclusion to a part with connected complement gives the displayed lower order bound. The only $1$-critical and $2$-critical graphs are $K_1$ and $K_2$; the latter has disconnected complement and therefore splits into singleton parts. Thus every non-singleton part has $r_i\ge3$.

A subdivision of $K_r$ replaces each edge of $K_r$ by a path, with the paths pairwise internally vertex-disjoint and with no branch vertex internal to another path. Restricting a drawing to such a subdivision and suppressing its degree-two vertices gives a drawing of $K_r$ with no additional crossings. This is why a subdivision yields a crossing-number lower bound directly, whereas a weak immersion requires smoothing.

We also use the small-excess subdivision theorem.

\begin{lemma}[Bar\'at--T\'oth; Luiz--Richter]\label{lem:small-excess}
If an $r$-critical graph has at most $r+5$ vertices, then it contains a subdivision of $K_r$.
\end{lemma}

The case up to $r+4$ is due to Bar\'at and T\'oth~\cite[Corollary~11]{BaratToth}; the case $r+5$ is due to Luiz and Richter~\cite[Theorem~1]{LuizRichter}.

\subsection{Complete and bipartite lower bounds}

Kleitman~\cite[Abstract]{Kleitman} proved, for every positive integer $t$,
\begin{equation}
\crn(K_{5,t})=4\floor{t/2}\floor{(t-1)/2}\ge t(t-2).
\label{eq:K5t}
\end{equation}
The last inequality is equality for even $t$; its two sides differ by one for odd $t$.

In a good drawing of $K_q$ with $q\ge7$, choose a uniform five-vertex set and retain only edges to its complement, of size $t=q-5\ge2$. Fix a crossing of edges $ab$ and $cd$. To retain both edges, the five-vertex side must contain exactly one of $a,b$ and exactly one of $c,d$. There are four such choices of two endpoints, and the other three vertices of that side can be chosen from the remaining $q-4$ vertices. Its survival probability is
\[
\frac{4\binom{q-4}{3}}{\binom q5}
=\frac{80t(t-1)}{(q)_4},\qquad (q)_4=q(q-1)(q-2)(q-3).
\]
The retained drawing has at least $\crn(K_{5,t})$ crossings. Taking expectations and using \eqref{eq:K5t} therefore gives
\begin{equation}
\crn(K_q)\ge C_q:=\frac{(q)_4(q-7)}{80(q-6)}\qquad(q\ge7).
\label{eq:Cq}
\end{equation}
We use the bound $C_q$ in the maximum-clique sampling estimate of Section~\ref{sec:middle}.

\begin{lemma}[Complete bipartite averaging]\label{lem:bipartite-average}
For integers $c\ge5$ and $t\ge1$,
\begin{equation}
\crn(K_{c,t})\ge\frac{c(c-1)}5
\floor{t/2}\floor{(t-1)/2}.
\label{eq:K6average}
\end{equation}
\end{lemma}
\begin{proof}
In a good drawing of $K_{c,t}$, sum \eqref{eq:K5t} over all choices of five vertices on the $c$-side. Every crossing has two distinct endpoints on that side and is counted $\binom{c-2}{3}$ times. Dividing the sum by that count gives
\[
\crn(K_{c,t})\ge
\frac{\binom c5}{\binom{c-2}3}\,4\floor{t/2}\floor{(t-1)/2}
=\frac{c(c-1)}5\floor{t/2}\floor{(t-1)/2}.
\]
\end{proof}

\section{Descent to a terminal critical graph}\label{sec:compression}

We derive edge lower bounds depending on order and clique number by iterated deletion of independent triples. These bounds will be used in the sampling argument for $n\le3r$. The proof has three steps: bound the edge loss in one deletion, sum these losses until no independent triple remains, and bound the edges of the terminal graph. The last subsection expresses the resulting bounds in terms of $r,n,w$.

\subsection{Two elementary deletion bounds}

Deletion of an independent triple from a critical graph lowers its chromatic number by one. Passing to a critical subgraph of the remainder may delete additional vertices and edges; the following estimates include this additional loss.

Let $H$ be $(s+1)$-critical and let $S$ be an independent set of three vertices. Criticality gives $\chi(H-S)\le s$. If $\chi(H-S)\le s-1$, giving all vertices of $S$ one new colour would colour $H$ with at most $s$ colours. Thus $\chi(H-S)=s$. Choose an $s$-critical subgraph $J\subseteq H-S$, and put
\[
U=V(H)\setminus V(J),\qquad |U|=3+e,
\qquad e\ge0.
\]
Write $a$ for the number of edges inside $U$ and $b$ for the number of edges between $U$ and $V(J)$. Since $J$ is a subgraph of $H[V(J)]$, we have $m(H)-m(J)\ge a+b$. The quantity $a+b$ counts all edges of $H$ incident with a vertex deleted from $H$ when retaining $J$. Any edges of $H[V(J)]$ omitted by $J$ increase the loss further, explaining why this is an inequality rather than an equality. Define
\[
Q_s(e)=\frac{e(2s-e-5)}2.
\]

Figure~\ref{fig:descent} gives an instance with $e=2$.
\begin{figure}[!htbp]
\centering
\scalebox{0.92}{%
\begin{tikzpicture}[x=1cm,y=1cm]
\node[proofpanel] at (0,.2) {(a) Critical descent: $C_7\vee K_4$ to $K_6$};
\foreach \i in {0,...,6}{
  \coordinate (a\i) at ({2.05+1.1*cos(90+360*\i/7)},{-2+1.1*sin(90+360*\i/7)});
}
\foreach \i/\j in {0/1,1/2,2/3,3/4,4/5,5/6,6/0}{\draw (a\i)--(a\j);}
\foreach \i in {0,2,4}{\node[proofvertex,fill=prooforange,draw=prooforange,label={90+360*\i/7}:$\i$] at (a\i) {};}
\foreach \i in {1,3}{\node[proofoutside,label={90+360*\i/7}:$\i$] at (a\i) {};}
\foreach \i in {5,6}{\node[proofbranch,label={90+360*\i/7}:$\i$] at (a\i) {};}
\draw[rounded corners,black!40] (.1,-.15) rectangle (4,-3.75);
\node[proofbox,text width=1.6cm] (kleft) at (2.05,-4.4) {$K_4$};
\draw[double,black!50] (2.05,-3.75)--(kleft.north);
\node[proofnote] at (2.05,-5.05) {$H=C_7\vee K_4$};
\draw[proofarrow] (4.2,-2)--node[above,font=\small,align=center]{delete $S$}(5.3,-2);
\foreach \x/\y/\i in {6/-1.4/1,8/-1.4/3,6.4/-2.5/5,7.6/-2.5/6}{\coordinate (b\i) at (\x,\y);}
\draw (b5)--(b6);
\foreach \i in {1,3}{\node[proofoutside,label=above:$\i$] at (b\i) {};}
\foreach \i in {5,6}{\node[proofbranch,label=below:$\i$] at (b\i) {};}
\draw[rounded corners,black!40] (5.55,-.65) rectangle (8.45,-3.35);
\node[proofbox,text width=1.6cm] (kmid) at (7,-4.4) {$K_4$};
\draw[double,black!50] (7,-3.35)--(kmid.north);
\node[proofnote] at (7,-5.05) {$H-S$, still $6$-chromatic};
\draw[proofarrow] (8.85,-2)--node[above,font=\small,align=center]{discard\\$1,3$}(10.6,-2);
\foreach \i in {1,...,6}{\coordinate (c\i) at ({12.2+1.1*cos(90+60*\i)},{-2+1.1*sin(90+60*\i)});}
\foreach \i in {1,...,6}{\foreach \j in {1,...,6}{\ifnum\i<\j\draw[black!35] (c\i)--(c\j);\fi}}
\foreach \i in {1,...,6}{\node[proofbranch] at (c\i) {};}
\node[proofnote,text width=4.2cm] at (12.2,-4.55) {$J=H[\{5,6\}\cup V(K_4)]$\\$J\cong K_6,\qquad e=2$};
\node[proofpanel] at (0,-5.6) {(b) Edge counts for the deletion};
\node[proofbox,text width=13.6cm] at (7.2,-6.9)
  {$U=\{0,1,2,3,4\}$ induces a path, so $a=4$.\\
   The cut has $5\cdot4$ edges to $K_4$ and the two edges $0\,6$, $4\,5$, so $b=22$.\\[3pt]
   $m(H)-m(J)=41-15=26=a+b,\qquad 2a+b=30=6(3+2).$};
\node[proofpanel] at (0,-8.35) {(c) The two general deletion bounds};
\node[proofbox,text width=6.3cm,minimum height=2.2cm] at (3.5,-9.95)
  {$2a+b\ge s(3+e)$,\quad $a\le\binom{3+e}{2}-3$\\[4pt]
   $\Longrightarrow\ a+b\ge3s+\dfrac{e(2s-e-5)}2$};
\node[proofbox,text width=6.3cm,minimum height=2.2cm] at (11,-9.95)
  {$b\ge3(s-e)$,\quad $a+b=\dfrac{(2a+b)+b}{2}$\\[4pt]
   $\Longrightarrow\ a+b\ge3s+\dfrac{s-3}{2}e$};
\node[proofnote,text width=14cm] at (7.25,-11.75)
  {Lower bounds: $23$ and $21$. Edge loss: $26$.\\
   After iteration: $|H_{\rm terminal}|=n-3(r-k)-E=3k-D-E$.};
\end{tikzpicture}
}%
\caption{Independent-triple deletion in $H=C_7\vee K_4$. Filled orange vertices form $S=\{0,2,4\}$; outlined squares denote the additional deletions $1,3$; blue vertices are retained. Double connectors denote complete joins. After deleting $S$, the cycle vertices $1,3$ have neighbours only in $K_4$, and their removal leaves $J\cong K_6$. Here $e=2$, $a=4$, $b=22$, and $m(H)-m(J)=26$, whereas Lemma~\ref{lem:deletion-bounds} gives the lower bounds $23$ and $21$.}
\label{fig:descent}
\end{figure}

\begin{lemma}[Deletion bounds]\label{lem:deletion-bounds}
With the notation above,
\begin{equation}
m(H)-m(J)\ge3s+Q_s(e),\qquad
m(H)-m(J)\ge3s+\frac{s-3}{2}e.
\label{eq:deletion-bounds}
\end{equation}
\end{lemma}
\begin{proof}
Every vertex of $H$ has degree at least $s$, so summing degrees over $U$ gives
\[
2a+b\ge s(3+e).
\]
The independent triple removes three possible edges from $H[U]$, hence
$a\le\binom{3+e}{2}-3$. Therefore
\begin{align*}
a+b&\ge s(3+e)-a\\
&\ge s(3+e)-\binom{3+e}{2}+3\\
&=3s+\frac{e(2s-e-5)}2.
\end{align*}
This proves the quadratic estimate.

Each vertex of $S$ has at most $e$ neighbours in $U$, because its other two vertices are nonadjacent to it. Each consequently has at least $s-e$ neighbours in $V(J)$. This inequality remains valid when $s-e$ is negative, and summing these three cross-degrees gives $b\ge3(s-e)$. Combining this with the degree sum yields
\[
a+b=\frac{(2a+b)+b}{2}
\ge\frac{s(3+e)+3(s-e)}2
=3s+\frac{s-3}{2}e.
\]
This proves the linear estimate.
\end{proof}

\subsection{Summing the deletion losses}

Deleting an independent triple changes the chromatic number and order from $(s+1,N)$ to $(s,N-3)$, leaving $3\chi-|V|$ unchanged because $3(s+1)-N=3s-(N-3)$. Each further vertex discarded in the passage to a critical subgraph increases $3\chi-|V|$ by one. Thus, with $D=3r-n$ and $E$ denoting the total number of these additional deletions, the terminal graph satisfies $3\chi(H)-|H|=D+E$.

\begin{corollary}[Descent bound]\label{cor:compressed-descent}
Let $G$ be $r$-critical of order $n$, and put
\begin{equation}
D=3r-n.
\label{eq:D}
\end{equation}
Repeatedly delete an independent triple and pass to a critical subgraph until a $k$-critical graph $H$ with $\alpha(H)\le2$ is obtained. If $E$ is the total number of vertices deleted beyond the mandatory three per step, define
\[
L_k(E)=\max\left\{\frac{E(2k-E-5)}2,\ \frac{k-3}{2}E\right\}.
\]
Then
\begin{equation}
m(G)\ge\frac32\bigl(r(r-1)-k(k-1)\bigr)+L_k(E)+m(H).
\label{eq:compressed}
\end{equation}
\end{corollary}
\begin{proof}
Every step lowers chromatic number by exactly one, as proved above. Order decreases, so the process terminates. Let $e_s\ge0$ be the additional deletion in the step from chromatic number $s+1$ to $s$, and put $E=\sum_{s=k}^{r-1}e_s$. Counting the remaining vertices gives
\[
|H|=n-3(r-k)-E=3k-D-E\ge k.
\]
Thus $E\le2k-D$. No lower bound on $D$ is needed to sum the deletion losses.

Sum the quadratic estimate in Lemma~\ref{lem:deletion-bounds}. Its additional contribution is
\begin{align*}
\sum_{s=k}^{r-1}Q_s(e_s)
&=\sum_{s=k}^{r-1}(s-5/2)e_s-\frac12\sum_{s=k}^{r-1}e_s^2\\
&\ge(k-5/2)E-\frac12E^2\\
&=Q_k(E).
\end{align*}
The inequality uses $s\ge k$ and
\[
E^2=\sum_s e_s^2+2\sum_{s<t}e_se_t\ge\sum_s e_s^2,
\]
since every $e_s$ is nonnegative. Independently, summing the linear estimate gives an additional contribution at least
\[
\sum_{s=k}^{r-1}\frac{s-3}{2}e_s\ge\frac{k-3}{2}E.
\]
Both estimates hold for the same descent sequence, so its total additional contribution is at least their maximum, $L_k(E)$. The mandatory contribution is
\[
3\sum_{s=k}^{r-1}s=\frac32\bigl(r(r-1)-k(k-1)\bigr).
\]
Adding $m(H)$ proves the stated bound. If no step occurs, $E=0$, both sums vanish, and $L_k(0)=0$, so that case is included.
\end{proof}

\subsection{The terminal complement}

For a critical graph with independence number at most two, Gallai's theorem and colouring after vertex deletion determine the orders of the critical join factors. Their triangle-free complements then yield an edge estimate.

\Needspace{7\baselineskip}

\begin{lemma}[Terminal factors]\label{lem:terminal-factors}
Let $H$ be $k$-critical with $\alpha(H)\le2$, and let $\ell$ be the number of connected components of $F=\overline H$. If $\omega(H)\le c$ for a positive integer $c$, then $\ell\le c$,
\begin{equation}
|H|=2k-\ell,\qquad
m(H)\ge\binom{2k-\ell}{2}-(k-\ell)(c-\ell+1)-\floor{(c-\ell+1)/2}.
\label{eq:terminal-factors}
\end{equation}
\end{lemma}
\begin{proof}
Let $F_1,\ldots,F_\ell$ be the complement components, and set $H_i=H[V(F_i)]$ and $k_i=\chi(H_i)$. Distinct components have no complement edges between them, so
\[
H=H_1\vee\cdots\vee H_\ell.
\]
The chromatic number of a join is the sum of the chromatic numbers of its factors: colours used in distinct factors are disjoint, and optimal factor colourings give a colouring of the join. Hence $k=\sum_i k_i$.

Each $H_i$ is $k_i$-critical. Indeed, if a proper subgraph $P\subset H_i$ had chromatic number $k_i$, then $P$ joined to the other factors would be a proper subgraph of $H$ with chromatic number $\sum_i k_i=k$, contradicting criticality.

The complement of $H_i$ is the connected graph $F_i$. By Gallai's theorem in Lemma~\ref{lem:gallai-join}, it cannot have at most $2k_i-2$ vertices. Thus $|H_i|\ge2k_i-1$. The case $k_i=1$ is simply a singleton. Conversely, deleting any vertex $v\in V(H_i)$ gives a colouring with at most $k_i-1$ colours. Every colour class has at most two vertices, since $\alpha(H_i-v)\le2$. Therefore
\[
|H_i|-1\le2(k_i-1),\qquad |H_i|\le2k_i-1.
\]
Both bounds give $|F_i|=|H_i|=2k_i-1$. Put $m_i=k_i-1$. Then
\begin{equation}
|F_i|=2m_i+1,\qquad
\sum_i m_i=k-\ell,\qquad |H|=2k-\ell.
\label{eq:terminal-orders}
\end{equation}

It remains to bound the complement's edges. The graph $F$ is triangle-free, because a complement triangle would be an independent triple of $H$. Let $a_i=\alpha(F_i)$. Every neighbourhood in $F_i$ is independent, so each vertex has degree at most $a_i$. Summing degrees gives
\[
m(F_i)\le\frac{|F_i|a_i}{2}=m_i a_i+\frac{a_i}{2}.
\]
Independence numbers add over components, and independent sets in $F$ are cliques in $H$. Consequently
\[
\sum_i a_i=\alpha(F)=\omega(H)\le c.
\]
Each $a_i\ge1$, giving $\ell\le c$. Let $I$ index the non-singleton components. For $i\in I$, we have $m_i\ge1$, and triangle-freeness gives $a_i\ge2$. Put
\[
p_i=a_i-1\ge1,\qquad
M=\sum_{i\in I}m_i=k-\ell,\qquad
P=\sum_{i\in I}p_i=\sum_i(a_i-1)\le c-\ell.
\]
Singleton components have $m_i=p_i=0$ and contribute no edges. They are excluded from the following sum, since $(m_i+1/2)(p_i+1)=1/2$ for such a component. If $I$ is empty, then $m(F)=0$ and the desired bound follows.

Otherwise, we combine the component bounds into one expression depending only on the totals $M$ and $P$. Define $\psi(m,p)=(m+1/2)(p+1)$. The degree bound gives $m(F_i)\le\psi(m_i,p_i)$. For positive integers $m,p,m',p'$, expansion gives
\[
\psi(m+m',p+p')-\psi(m,p)-\psi(m',p')
=mp'+m'p-\frac12\ge\frac32>0.
\]
For $|I|=1$ the second inequality below is equality. For $|I|\ge2$, repeated application of this identity gives that inequality, since every partial sum of the pairs $(m_i,p_i)$ has positive integer entries. Thus
\[
m(F)\le\sum_{i\in I}\psi(m_i,p_i)
\le\psi(M,P)
\le\left(k-\ell+\frac12\right)(c-\ell+1).
\]
The last inequality uses $P\le c-\ell$ and $M=k-\ell\ge0$. The integral edge count consequently satisfies
\[
m(F)\le(k-\ell)(c-\ell+1)+\lfloor(c-\ell+1)/2\rfloor.
\]
Subtracting this from $\binom{|H|}{2}$ and using \eqref{eq:terminal-orders} proves the lemma.
\end{proof}

Figure~\ref{fig:terminal} illustrates the terminal decomposition and the two order bounds.
\begin{figure}[!htbp]
\centering
\scalebox{0.90}{%
\begin{tikzpicture}[x=1cm,y=1cm]
\node[proofpanel] at (0,0) {(a) $H=C_5\vee C_5\vee K_1$};
\begin{scope}[yshift=-.9cm]
\foreach \cx/\name in {1.8/L,6.4/M}{
  \draw[black!35,rounded corners] (\cx-1.35,-.65) rectangle (\cx+1.35,-3.7);
  \foreach \i in {1,...,5}{\node[proofvertex] (\name\i) at ({\cx+cos(90+72*\i)},{-2+sin(90+72*\i)}) {};}
  \foreach \i/\j in {1/2,2/3,3/4,4/5,5/1}{\draw (\name\i)--(\name\j);}
  \node[proofnote] at (\cx,-3.35) {$k_i=3$, $|H_i|=5$};
}
\node[proofbranch] (one) at (11,-2) {};
\node[proofnote] at (11,-3.35) {$k_i=1$, $|H_i|=1$};
\draw[double,black!55] (3.2,-1.7)--node[above,font=\small] {join}(5,-1.7);
\draw[double,black!55] (7.8,-1.7)--node[above,font=\small] {join}(10.7,-1.7);
\draw[double,black!55] (1.8,-.55) to[bend left=18] node[above,font=\small] {join} (11,-.55);
\node[proofnote,text width=3cm] at (13.3,-2) {Double connectors:\\complete joins\\between factors.};
\node[proofnote] at (7.5,-4.25) {$k=3+3+1=7,\quad \ell=3,\quad |H|=11=2k-\ell,\quad \omega(H)=5.$};
\node[proofpanel] at (0,-5.05) {(b) Complement components};
\foreach \cx/\name in {2.2/F,7.2/T}{
  \draw[black!35,rounded corners] (\cx-1.4,-5.5) rectangle (\cx+1.4,-8.8);
  \foreach \i in {1,...,5}{\node[proofvertex] (\name\i) at ({\cx+cos(90+72*\i)},{-6.95+sin(90+72*\i)}) {};}
  \foreach \i/\j in {1/3,3/5,5/2,2/4,4/1}{\draw[proofblue] (\name\i)--(\name\j);}
  \node[proofnote] at (\cx,-8.35) {$F_i=\overline{H_i}\cong C_5$};
}
\node[proofvertex] at (12.2,-6.95) {};
\node[proofnote] at (12.2,-8.35) {$F_3=K_1$};
\node[proofnote,text width=2cm] at (4.7,-6.9) {no edges\\between\\components};
\node[proofnote,text width=2cm] at (9.7,-6.9) {no edges\\between\\components};
\node[proofpanel] at (0,-9.5) {(c) Order and degree bounds};
\node[proofbox,text width=6.4cm,minimum height=1.6cm] (low) at (3.6,-10.85)
  {Connected $\overline{H_i}$ and criticality\\$\xRightarrow{\text{Gallai}} |H_i|\ge2k_i-1$};
\node[proofbox,text width=6.4cm,minimum height=1.6cm] (up) at (11.4,-10.85)
  {Colour $H_i-v$ with $k_i-1$ colours.\\Each class has at most two vertices:\\$|H_i|-1\le2(k_i-1)$.};
\node[proofbox,text width=12.5cm,draw=proofblue] (equal) at (7.5,-13)
  {$|H_i|=2k_i-1\quad\Longrightarrow\quad |H|=2k-\ell.$\\[3pt]
   $F_i$ is triangle-free, so $N_{F_i}(v)$ is independent and $\deg_{F_i}(v)\le\alpha(F_i)$.};
\draw[proofarrow] (low.south)--++(0,-.25)-|(equal.north);
\draw[proofarrow] (up.south)--++(0,-.25)-|(equal.north);
\end{scope}
\end{tikzpicture}
}%
\caption{The terminal graph $H=C_5\vee C_5\vee K_1$ and its complement, on the same vertex sets. Double connectors denote complete joins; unmarked edge intersections are not vertices. Panel (c) records the two order bounds and the triangle-free degree estimate in Lemma~\ref{lem:terminal-factors}.}
\label{fig:terminal}
\end{figure}

\FloatBarrier
\subsection{Polynomial edge bounds}

For the next theorem, let $w$ be an integer upper bound on $\omega(G)$. The terminal graph $H$ is a subgraph of $G$ and has chromatic number $k$. Therefore $\omega(H)\le\omega(G)\le w$ and $\omega(H)\le k$, so Lemma~\ref{lem:terminal-factors} applies with $c=\min\{k,w\}$. The two nonnegative gaps $z=k-c$ and $q=w-c$ cannot both be positive: $zq=0$. This identity eliminates the mixed terms in the calculations below. The resulting bounds depend only on $r,n,w$. Here $c=\min\{k,w\}$ bounds the terminal clique number; in Section~\ref{sec:near}, $c$ denotes the order excess $n-r$.

\Needspace{12\baselineskip}

\begin{theorem}[Clique-sensitive edge bounds]\label{thm:clique-edge}
Let $G$ be $r$-critical of order $n$, with $\omega(G)\le w<r$ for a positive integer $w$. Put $D=3r-n\ge0$. Then $D\le r+w$ and
\begin{equation}
m(G)\ge\frac{3r^2-D^2+2Dw-2w^2}{2}-3r.
\label{eq:quadratic-edge}
\end{equation}
Moreover,
\begin{align}
m(G)&\ge\frac{3r^2+D^2-6Dw+6w^2}{2}-3r
\quad(D\ge2w),\label{eq:third-edge}\\
m(G)&\ge\frac{3r(r-1)}2-\frac{wD}2-2w+\frac{3D}2-\frac32
\quad(D\le w).\label{eq:first-edge}
\end{align}
\end{theorem}
\begin{proof}
Let $H$ be the terminal $k$-critical graph from Corollary~\ref{cor:compressed-descent}, and let $E$ be the total additional deletion. Since $H\subseteq G$ and $\omega(H)\le\chi(H)=k$, put $c=\min\{k,w\}$, so $\omega(H)\le c$, and let $\ell$ count the connected components of $\overline H$. Lemma~\ref{lem:terminal-factors} and the descent order give
\[
2k-\ell=|H|=3k-D-E,\qquad E=k+\ell-D\ge0.
\]
Consequently $1\le\ell\le c\le w$, $k\le r$, and $D\le k+\ell\le r+w$. Define
\[
z=k-c\ge0,\qquad q=w-c\ge0.
\]
Since $c$ is the minimum of $k$ and $w$, at least one gap is zero; hence $qz=0$. Also $k=c+z$ and $r\ge2$.

Write $A_j=3j(j-1)/2$. Since $\lfloor(c-\ell+1)/2\rfloor\le(c-\ell+1)/2$, Lemma~\ref{lem:terminal-factors} gives the weaker lower bound $m(H)\ge M_H$, where
\[
M_H=\binom{2k-\ell}{2}-(k-\ell)(c-\ell+1)-\frac{c-\ell+1}{2}.
\]
The two descent estimates therefore give $m(G)\ge\Xi_L$ and $m(G)\ge\Xi_Q$, where
\begin{align}
\Xi_L&=A_r-A_k+\frac{k-3}{2}E+M_H,\notag\\
\Xi_Q&=A_r-A_k+\frac{E(2k-E-5)}2+M_H.
\label{eq:edge-expressions}
\end{align}
Let $R$ and $Q_2$ denote the right-hand sides of \eqref{eq:first-edge} and \eqref{eq:quadratic-edge}, respectively. We compare the linear descent bound $\Xi_L$ with $R$ and the quadratic descent bound $\Xi_Q$ with $Q_2$. Instead of estimating $k,E,\ell$ separately, we substitute $E=k+\ell-D$ and expand the two differences. In the gap variables $z,q$, the mixed terms vanish by $qz=0$, while the remaining contributions can be controlled in the regimes $D\le w$ and $D>w$.

Expanding the expression for $M_H$ gives
\[
2M_H=4k^2-2k\ell-\ell^2-4k-2kc+2c\ell-c+4\ell-1.
\]
Substituting this expression and $E=k+\ell-D$ into $\Xi_L$ gives
\begin{align*}
2(\Xi_L-R)
&=2k^2-2kc-k\ell-\ell^2+2c\ell-c+\ell+(w-k)(D+4)+2\\
&=2(z-1)^2+(c-\ell)(\ell-1)+(2c-\ell-D)z+q(D+4)\\
&=2(z-1)^2+(c-\ell)(\ell+z-1)+(w-D)z+q(D+4)-qz.
\end{align*}
The second line uses $k=c+z$ and $w=c+q$; in the last line, $2c-\ell-D=(c-\ell)+(w-D)-q$.

For the quadratic expression,
\[
2(\Xi_Q-\Xi_L)=E(k-E-2)=E(D-\ell-2).
\]
Also $2(R-Q_2)=3r+(D-w)(D-2w)+3D-4w-3$. Hence
\begin{align*}
2(\Xi_Q-Q_2)
&=2(\Xi_L-R)+3r+(D-w)(D-2w)\\
&\quad+3D-4w-3+E(D-\ell-2)\\
&=3r+5D-8w-1+2z^2-6z+7q\\
&\quad+2(w-\ell)(E+q+1/2)-2qz.
\end{align*}
Here the last equality follows by inserting the linear expansion above and $E=c+z+\ell-D$. These exact identities are the algebraic inputs for the comparison below; their signs and combined lower bounds are handled in the indicated ranges. Since $qz=0$, division by two gives
\begin{align}
\Xi_L={}&R+(z-1)^2+\frac{(c-\ell)(\ell+z-1)}2\notag\\*
&+\frac{(w-D)z}{2}+\frac{q(D+4)}2,\label{eq:edge-linear-identity}\\
\Xi_Q={}&Q_2+\frac{3r}{2}+\frac{5D}{2}-4w-\frac12+z^2-3z\notag\\*
&+\frac{7q}{2}+(w-\ell)(E+q+1/2).
\label{eq:edge-quadratic-identity}
\end{align}
We apply the linear identity for $D\le w$ and the quadratic identity for $D>w$. In the first regime, the correction to $R$ has termwise nonnegative sign, after which $R\ge Q_2$ supplies the general quadratic bound. In the second regime, we bound the only potentially negative quadratic contribution, $z^2-3z$, and obtain both $Q_2$ and its improvement when $D\ge2w$.

\paragraph{The linear estimate: $D\le w$.}
Every term after $R$ in \eqref{eq:edge-linear-identity} is nonnegative: $c\ge\ell\ge1$, $z,q\ge0$, $D\ge0$, and $w-D\ge0$. Thus $m(G)\ge R$, proving \eqref{eq:first-edge}.

This also proves the quadratic bound in this regime. Put $h=w-D\ge0$. Substitution and completion of the square give
\begin{align*}
R-Q_2&=h^2+h(D/2-2)+\frac{3r-D-3}{2}\\
&=(h-1)^2+\frac{Dh}{2}+\frac{3r-D-5}{2}\\
&\ge r-2\ge0.
\end{align*}
The last line uses $D\le w\le r-1$ and $D,h\ge0$.

\paragraph{The quadratic estimate: $D>w$.}
The last two terms in \eqref{eq:edge-quadratic-identity} are nonnegative. Put $\delta=\max\{D-2w,0\}$. Since
\[
D\le k+\ell=c+z+\ell\le2w+z,
\]
we have $z\ge\delta$. With the nonnegative integer $y=z-\delta$,
\begin{align*}
z^2-3z&=\delta^2-3\delta+(y-1)(y-2)-2+2\delta y\\
&\ge\delta^2-3\delta-2.
\end{align*}
Here $(y-1)(y-2)\ge0$ for $y=0,1$ and for every integer $y\ge2$, while $\delta y\ge0$.

The bounds $D\le r+w$, $D\ge w+1$, and $w\ge1$ also give
\[
0\le\delta\le r-w,\qquad0\le\delta\le D-w-1,
\]
since both upper bounds are nonnegative and at least $D-2w$. Substitution in \eqref{eq:edge-quadratic-identity} gives
\begin{align*}
m(G)-(Q_2+\delta^2)
&\ge\frac{3r}{2}+\frac{5D}{2}-4w-\frac52-3\delta\\*
&=(r-w)+\frac{r-w-\delta}{2}+\frac52(D-w-1-\delta)\\
&\ge r-w>0.
\end{align*}
Hence $m(G)\ge Q_2+\delta^2\ge Q_2$, proving \eqref{eq:quadratic-edge}. If $D\ge2w$, then $\delta=D-2w$ and expansion gives
\[
Q_2+\delta^2=\frac{3r^2+D^2-6Dw+6w^2}{2}-3r,
\]
which proves \eqref{eq:third-edge}, including $D=2w$. All three bounds follow.
\end{proof}

\FloatBarrier
\section{The near range}\label{sec:near}

We now assume $r\ge19$ and $n=r+c<5r/4$. Thus
\begin{equation}
0\le c<r/4.
\label{eq:near-c}
\end{equation}
If $c\le5$, the published small-excess theorem gives a subdivision of $K_r$, so the conjecture holds. We may therefore assume $c\ge6$. In particular $r\ge25$, and $r+c\le2r-2$, allowing Gallai's decomposition.

We construct a weak immersion of $K_r$ whose path union is edge-disjoint from a complete bipartite graph. Crossings within the reserved graph compensate for the increase caused by smoothing the immersion. A route is \emph{nontrivial} if it has more than one edge; the smoothing bound depends on the number of these routes and their incidences at internal vertices.

\subsection{Choosing the branch sets}

Write the Gallai decomposition as
\[
G=G_1\vee\cdots\vee G_t\vee K_s,
\qquad \chi(G_i)=k_i,
\qquad |V(G_i)|=k_i+c_i,
\]
where the $G_i$ are non-singleton factors and $K_s$ is the join of the $s$ singleton factors. We will choose $k_i$ branch vertices in each $G_i$ and use every singleton as a branch vertex. Then $k_i\ge3$, $c_i\ge k_i-1$, and subtracting the chromatic-number sum from the order sum gives $\sum_i c_i=c$.

\begin{lemma}[Few missing branch pairs]\label{lem:few-routes}
For each $i$, a $k_i$-vertex set $U_i\subseteq V(G_i)$ can be chosen so that the total number $\tau$ of nonedges inside these sets satisfies
\begin{equation}
\tau\le T_c:=\floor{c(c+1)/4}.
\label{eq:tau}
\end{equation}
\end{lemma}
\begin{proof}
Every vertex of the $k_i$-critical graph $G_i$ has at least $k_i-1$ neighbours, and therefore at most $c_i$ non-neighbours. Its complement has at most $(k_i+c_i)c_i/2$ edges. Choose $U_i$ uniformly from its $k_i$-subsets. Each complement edge has probability $\binom{k_i}{2}/\binom{k_i+c_i}{2}$ of having both endpoints selected. Hence the expected number of missing pairs is at most
\[
\frac{c_i k_i(k_i-1)}{2(k_i+c_i-1)}
\le\frac{c_i k_i}{4}
\le\frac{c_i(c_i+1)}4.
\]
The first inequality uses $k_i+c_i-1\ge2(k_i-1)$; the second uses $k_i\le c_i+1$. Choose the factor sets independently. Their expected total is at most
\[
\frac14\sum_i(c_i^2+c_i)
\le\frac14\left(\left(\sum_i c_i\right)^2+\sum_i c_i\right)
=\frac{c(c+1)}4.
\]
Some choice attains at most this expectation. Since the count is integral, the stated bound follows.
\end{proof}

\subsection{Reserved routing}

The following construction uses the Gallai--Shannon framework of Fox, Pach and Suk~\cite[Section~2, proof of Theorem~2]{FoxPachSuk}, with a common set of routing singletons and a disjoint set of reserved singletons. A missing pair $u,u'\in U_i$ is connected in one of the two forms
\[
u-x-u' \qquad\text{or}\qquad u-x-p-y-u',
\]
where $x,y\in W_i:=V(G_i)\setminus U_i$ and $p$ is a routing singleton. Injective entrance assignments distinguish the edges $ux$ and $yu'$; an auxiliary edge-colouring assigns the routing vertices so that the middle edges are also distinct. The reserved singletons are excluded from this routing pool. Both routing singletons and reserved singletons remain branch vertices of the immersed $K_r$. Weak immersion permits a routing singleton to be an internal vertex of another route. The reserve consists of edges from nonbranch vertices to reserved singletons, so using a reserved singleton as an endpoint does not use a reserved edge.

An instance of the construction is given in Figure~\ref{fig:routing}.

\begin{figure}[!htbp]
\centering
\scalebox{0.92}{%
\begin{tikzpicture}[x=1cm,y=1cm]
\node[proofpanel] at (0,0) {(a) $G=C_5\vee C_5\vee C_5\vee K_{16}$: $r=25$, $n=31$, $c=6$};
\foreach \cx/\id in {2.1/x,7.3/y,12.5/z}{
  \foreach \i in {0,...,4}{
    \coordinate (\id\i) at ({\cx+1.05*cos(90+72*\i)},{-2+1.05*sin(90+72*\i)});
  }
  \foreach \i/\j in {0/1,1/2,2/3,3/4,4/0}{\draw[black!65] (\id\i)--(\id\j);}
  \foreach \i in {0,...,4}{\node[proofvertex,label={90+72*\i}:$\id_{\i}$] at (\id\i) {};}
}
\foreach \i in {0,2,4}{\node[proofbranch] at (x\i) {};}
\foreach \id in {y,z}{\foreach \i in {0,1,2}{\node[proofbranch] at (\id\i) {};}}
\node[proofnote] at (2.1,-3.6) {$U_1=\{x_0,x_2,x_4\}$};
\node[proofnote] at (7.3,-3.6) {$U_2=\{y_0,y_1,y_2\}$};
\node[proofnote] at (12.5,-3.6) {$U_3=\{z_0,z_1,z_2\}$};
\node[proofbox,text width=13.6cm] at (7.3,-4.6)
  {$K_{16}$: a routing pool $P$ of $4$ vertices, a reserve $B$ of $8$, and $4$ others.\\
   All $16$ are branch vertices. Every pair of different factors is completely joined.};
\node[proofpanel] at (0,-5.65) {(b) Routes for the four missing branch pairs};
\node[proofnote,anchor=west,text width=6.7cm] at (0,-6.6)
  {In the first cycle:\\$x_0-x_1-x_2,\qquad x_2-x_3-x_4$.\\
   At $x_2$, the entrances $x_1,x_3$ are distinct.};
\node[proofnote,anchor=west,text width=7.2cm] at (7.3,-6.6)
  {In the other two cycles, using one $p\in P$:\\
   $y_0-y_4-p-y_3-y_2$,\\
   $z_0-z_4-p-z_3-z_2$.};
\node[proofpanel] at (0,-7.65) {(c) One auxiliary colour and a common routing vertex $p$};
\node[proofvertex,label=left:$y_4$] (ya) at (.7,-8.6) {};
\node[proofvertex,label=right:$y_3$] (yb) at (3.3,-8.6) {};
\node[proofvertex,label=left:$z_4$] (za) at (.7,-9.65) {};
\node[proofvertex,label=right:$z_3$] (zb) at (3.3,-9.65) {};
\draw[proofroute] (ya)--node[above,font=\small]{$p$}(yb);
\draw[proofroute] (za)--node[above,font=\small]{$p$}(zb);
\draw[proofarrow] (4.2,-9.1)--(5.4,-9.1);
\node[proofbranch,label=above:$p$] (p) at (7.4,-9.1) {};
\foreach \x/\y/\lab in {5.9/-8.5/y_4,8.9/-8.5/y_3,5.9/-9.8/z_4,8.9/-9.8/z_3}{
  \node[proofvertex] (end\lab) at (\x,\y) {};
  \node[proofnote,above] at (\x,\y) {$\lab$};
  \draw[proofroute] (end\lab)--(p);
}
\node[proofnote,text width=4.9cm] at (12.2,-9.1)
  {The auxiliary edges are disjoint.\\
   The routes are edge-disjoint\\and have two internal\\passages through $p$.};
\node[proofpanel] at (0,-10.65) {(d) The $48$ reserved edges between $W$ and $B$};
\foreach \i/\lab in {1/x_1,2/x_3,3/y_3,4/y_4,5/z_3,6/z_4}{
  \coordinate (w\i) at ({.7+.72*\i},-11.7);
  \node[proofvertex,label=above:$\lab$] at (w\i) {};
}
\foreach \i in {1,...,8}{
  \coordinate (b\i) at ({.65+.6*\i},-13.15);
  \node[proofbranch] at (b\i) {};
}
\foreach \i in {1,...,6}{\foreach \j in {1,...,8}{\draw[prooforange!55,densely dashed,line width=.3pt] (w\i)--(b\j);}}
\node[proofnote] at (3.4,-13.6) {$B$: the eight reserved singletons};
\node[proofbox,text width=7.1cm] at (11,-12.4)
  {$W=\{x_1,x_3,y_3,y_4,z_3,z_4\}$.\\
   Nontrivial routes use $U_i$--$W_i$ and $P$--$W$.\\
   Direct routes join branch vertices.\\
   Reserved edges join $B$ to $W$, giving $K_{6,8}$.};
\end{tikzpicture}
}%
\caption{The construction of Lemma~\ref{lem:reserved} on the $25$-critical graph $C_5\vee C_5\vee C_5\vee K_{16}$, with $n=31<5r/4$. Filled cycle vertices are branches. The four missing branch pairs give $\tau=4\le T_6=10$; all other branch pairs use their original edges. The prescribed sets have sizes $|P|=\lfloor3\max k_i/2\rfloor=4$ and $b=25-7-10=8$, although the auxiliary matching uses only one colour. Panel (d) shows the unused $K_{6,8}$. The layouts specify graph incidences; their crossings are not used in the crossing estimate.}
\label{fig:routing}
\end{figure}

\FloatBarrier
\begin{lemma}[Reserved routing]\label{lem:reserved}
Let $G$ be $r$-critical on $r+c\le2r-2$ vertices, where $c\ge1$, and put
\begin{equation}
b=r-(c+1)-\floor{3(c+1)/2}.
\label{eq:breserved}
\end{equation}
If $b\ge0$, the chosen branch sets admit a weak immersion of $K_r$ with at most $T_c$ nontrivial routes, and none of its paths uses an edge of a copy of $K_{c,b}$.
\end{lemma}
\begin{proof}
Use the Gallai factors introduced above, and choose the sets $U_i$ from Lemma~\ref{lem:few-routes}. Since $c\ge1$, at least one factor is non-singleton. Put $W_i=V(G_i)\setminus U_i$ and $W=\bigcup_i W_i$. The sets $U_i$, together with all singleton factors, give $s+\sum_i k_i=r$ branch vertices; $|W|=\sum_i c_i=c$. Distinct factors are joined completely, so the only missing branch pairs are the nonedges within the $U_i$. Their number is at most $T_c$.

\paragraph{Private entrance edges.}
Fix $u\in U_i$ and let $A_u$ be its non-neighbours in $U_i$. If $|A_u|=a$, then $u$ has at most $k_i-1-a$ neighbours in $U_i$. Minimum degree $k_i-1$ therefore gives at least $a$ neighbours in $W_i$. Assign a different such neighbour to each member of $A_u$, obtaining an injection
\[
f_u:A_u\longrightarrow N(u)\cap W_i.
\]
For a missing pair $uu'$, put $x=f_u(u')$ and $y=f_{u'}(u)$. If $x=y$, route it by $u-x-u'$. Otherwise record an auxiliary edge $xy$ on $W$, labelled by the missing pair $uu'$. Let $J$ be the multigraph containing all these auxiliary edges from all factors. It has no loops; parallel edges are allowed and retain their distinct labels. An auxiliary edge records a routing demand; it need not be an edge of $G$. Its colour will specify the two-edge connection $x-p-y$ used to complete that demand.

For $x\in W_i$, each branch vertex in $U_i$ contributes at most one auxiliary incidence at $x$, by its injection. Auxiliary edges from other factors have endpoints in their own, disjoint $W_j$. Thus
\[
\Delta(J)\le q:=\max_i k_i.
\]
Shannon's theorem~\cite[main theorem]{Shannon} (see also~\cite[Section~2, p.~50:3]{FoxPachSuk}) states that a loopless multigraph of maximum degree at most $q$ has a proper edge-colouring with at most $\lfloor3q/2\rfloor$ colours. Apply it to this single multigraph $J$. A colour class is a matching, so its auxiliary edges have pairwise distinct endpoints throughout $W$. The multigraph $J$ is coloured once across all non-singleton factors. Thus the same routing singleton can serve different factors without any middle edge being reused.

\paragraph{A common pool of routing singletons.}
Write $K=\sum_i k_i$. Of the $r$ branch vertices, $K$ belong to the non-singleton factors. We need one singleton for each auxiliary colour and $b$ additional singletons for the reserved set. The following inequality shows that there are enough of the remaining $r-K$ singleton branches for both sets:
\[
K+\lfloor3q/2\rfloor+b\le r.
\]
For each non-singleton factor, $c_i\ge k_i-1\ge2$, so $k_i\le c_i+1\le3c_i/2$. Choose $j$ with $k_j=q$. Then
\[
K+\frac32q=\frac52q+\sum_{i\ne j}k_i
\le\frac52(c_j+1)+\frac32(c-c_j)
\le\frac52(c+1).
\]
Since $K$ is integral, taking floors gives
\[
K+\floor{3q/2}\le(c+1)+\floor{3(c+1)/2}.
\]
There are $r-K$ singleton factors. The definition of $b$ therefore leaves at least $\lfloor3q/2\rfloor$ routing singletons and $b$ further singletons. Select a pool $P$ of the former size, rename the auxiliary colours by vertices of $P$, and let $B$ be a set of $b$ other singleton vertices. Any remaining singletons are also branch vertices but are not used internally by a route. If auxiliary edge $xy$ has colour $p$, route its missing pair by
\[
u-x-p-y-u'.
\]
The entrance edges exist by the injections, and the middle edges exist because a singleton factor is adjacent to every other factor. Use the original edge for each adjacent branch pair.

\paragraph{Edge-disjointness and the reserve.}
Every nontrivial route uses entrance/exit edges between $U_i$ and $W_i$, and, when $x\ne y$, middle edges between $P$ and $W$. Entrance edges cannot repeat: at a fixed branch vertex $u$, the injection $f_u$ assigns different neighbours to distinct missing pairs. Each such edge has exactly one endpoint in $\bigcup_i U_i$, so assignments at different branch vertices cannot reuse it either. Middle edges cannot repeat: reuse of $px$ would mean two auxiliary edges of colour $p$ share $x$, contradicting the proper colouring of $J$. This argument covers all factors at once. The two edge classes are disjoint because $P$ consists of singleton factors and is disjoint from every $U_i$. A direct branch edge has both endpoints among the branch vertices, whereas every nontrivial-route edge touches $W$; hence direct edges do not meet either class of used edges. All routed paths are therefore pairwise edge-disjoint.

Finally, every vertex of $W$ is adjacent to each of the $b$ reserved singletons. No route uses one of these edges, since its middle singleton belongs to $P$. These $cb$ unused edges form the required $K_{c,b}$.
\end{proof}

\FloatBarrier
\subsection{Counting the smoothing cost}

Routes may share a nonbranch vertex or pass internally through a branch vertex. Following the local smoothing argument of Fox, Pach and Suk~\cite[proof of Lemma~8]{FoxPachSuk}, we reconnect the paths inside disjoint vertex discs, preserving all crossings outside the discs. The counts below use the number of nontrivial routes and their passages in the reserved construction.

\begin{lemma}[Path-count smoothing]\label{lem:smoothing}
For the reserved immersion of Lemma~\ref{lem:reserved}, any good drawing of its path union can be converted into a drawing of $K_r$ by adding at most
\begin{equation}
\Lambda(r,c):=\frac{2r+5c-4}{4}\floor{c(c+1)/4}
\label{eq:smoothing-loss}
\end{equation}
crossings.
\end{lemma}
\begin{proof}
Let $\tau$ be its number of nontrivial routes. Choose pairwise disjoint discs around all vertices, containing no original crossing and meeting each incident edge once. Inside each disc, straighten the incident arcs into radial strands while preserving their cyclic order; this creates no crossing. All subsequent changes are confined to these discs, so crossings outside them are retained.

\paragraph{Nonbranch vertices.} First consider a vertex $v\in W_i$. If $f(v)$ paths pass through it, reconnect each path's two boundary endpoints by a chord. Two chords cross at most once; a small perturbation makes any coincident intersections pairwise without increasing the number of intersecting pairs. The cost is at most $\binom{f(v)}2$.

A two-edge route through $v$ uses two edges to $U_i$, and a four-edge route through $v$ uses one. Thus each passage uses at least one entrance or exit edge towards $U_i$. The injective assignments in the routing proof make these edges distinct for different paths, so $f(v)\le k_i$. Consequently
\[
\binom{f(v)}2\le\frac{k_i-1}{2}f(v)\le\frac{c_i}{2}f(v).
\]
A two-edge route visits one nonbranch vertex and a four-edge route visits two. Thus $\sum_{v\in W}f(v)\le2\tau$. Using $c_i\le c$, the total nonbranch cost is at most $c\tau$.

\paragraph{Routing branch vertices.} An internal branch passage occurs only at a routing singleton $p\in P$. Its $f_p$ passages correspond to auxiliary edges of colour $p$. Those edges form a matching on $W$, so their $2f_p$ endpoints are distinct and $2f_p\le|W|=c$. The $r-1$ paths ending at $p$ each contribute one incident strand, while each internal passage contributes two. These $r-1+2f_p$ strands use distinct edges by edge-disjointness. Branch vertices outside $P$ have no internal passages and require no modification.
Process these passages in successively smaller concentric discs. At each stage, the unprocessed strands are radial inside the current disc, in the original cyclic order with the processed passage endpoints omitted; all earlier bypasses and their intersections lie outside that disc. When $h$ passages remain, there are $r-1+2h$ radial strands. Remove the two radial strands of the current passage, retaining their endpoints on the disc boundary. There are now $r-3+2h$ other radial strands. One of the two boundary arcs between them contains at most half the remaining radial endpoints. A bypass close to that arc therefore crosses at most
\[
\floor{(r-3)/2}+h
\]
remaining strands, each once, while avoiding the centre. Here a bypass is the new connecting arc inside the current disc; the outer portions of its two strands are retained.

After inserting a bypass, choose the next disc with radius smaller than its distance from the centre and the corresponding distances of all earlier bypasses. These distances are positive by compactness. The new disc meets every unprocessed radial strand once, in the inherited cyclic order, and contains no earlier bypass or crossing. Every subsequent modification is confined to this smaller disc. Thus newly inserted connecting arcs lie in disjoint annuli, and no subsequent step creates a new intersection with an earlier connecting arc.

An earlier bypass may already cross an outer strand of a passage processed later. That crossing remains after the later passage is reconnected, and was counted when the earlier bypass was inserted. It is not counted again. This accounts for crossings between completed passages, including passages whose endpoints alternate around the original disc, and justifies summing the stage bounds below.

Summing over $h=f_p,f_p-1,\ldots,1$ gives a cost at $p$ of at most
\[
\floor{(r-3)/2}f_p+\frac{f_p(f_p+1)}2
\le\frac{r+f_p-2}{2}f_p
\le\frac{2r+c-4}{4}f_p.
\]
\paragraph{Total cost.} Each four-edge route has one such branch passage, so $\sum_{p\in P}f_p\le\tau$. Since the vertex discs are disjoint, the local crossing counts can be added. The entire added cost is at most
\[
c\tau+\frac{2r+c-4}{4}\tau
=\frac{2r+5c-4}{4}\tau
\le\Lambda(r,c).
\]
Finally, a routed curve can inherit self-intersections from nonconsecutive edges of its original path. Subdivide its trace at these points and retain a simple path between the same endpoints. Removing the other loops only deletes segments and adds no crossing. The resulting curves form a drawing of $K_r$.
\end{proof}

The local modifications are shown in Figure~\ref{fig:smoothing}.

\begin{figure}[!htbp]
\centering
\scalebox{0.85}{%
\begin{tikzpicture}[x=1cm,y=1cm]
\node[proofpanel] at (0,0) {(a) Reconnection at a nonbranch vertex};
\foreach \cx/\id in {2/A,7/B}{
  \draw[black!45] (\cx,-2) circle (1.2);
  \foreach \a in {155,90,20,-20,-105,-160}{
    \coordinate (\id\a) at ({\cx+1.2*cos(\a)},{-2+1.2*sin(\a)});
    \fill (\id\a) circle (1.6pt);
  }
}
\foreach \a/\b/\col in {155/-20/proofblue,90/-105/prooforange,20/-160/proofteal}{
  \draw[\col,line width=1pt] (A\a)--(2,-2)--(A\b);
  \draw[\col,line width=1pt] (B\a)--(B\b);
}
\node[proofvertex,fill=black,label=below:$v$] at (2,-2) {};
\draw[proofarrow] (3.6,-2)--node[above,font=\small]{reconnect}(5.4,-2);
\node[proofnote,text width=4.5cm] at (12,-2)
  {$f(v)$ passages become chords.\\Any two chords cross at most once:\\[4pt]
   new crossings $\le\binom{f(v)}2$.};
\node[proofnote] at (2,-3.55) {Passages share a graph vertex.};
\node[proofnote] at (7,-3.55) {The vertex is removed.};
\node[proofpanel] at (0,-4.35) {(b) Successive bypasses at a branch vertex};
\foreach \cx/\id in {2/C,7/D,12/E}{
  \draw[black!45] (\cx,-6.35) circle (1.3);
  \foreach \i in {0,...,23}{
    \pgfmathsetmacro{\a}{7.5+15*\i}
    \draw[black!55] (\cx,-6.35)--++(\a:1.3);
  }
  \node[proofbranch,label=below right:$p$] at (\cx,-6.35) {};
}
\foreach \a in {18,162}{\draw[proofroute] (2,-6.35)--++(\a:1.3);}
\foreach \a in {198,342}{\draw[prooforange,line width=1pt,dashed] (2,-6.35)--++(\a:1.3);}
\foreach \a in {198,342}{\draw[prooforange,line width=1pt,dashed] (7,-6.35)--++(\a:1.3);}
\foreach \cx in {7,12}{
  \draw[proofroute] ({\cx+1.3*cos(18)},{-6.35+1.3*sin(18)})
    --({\cx+1.05*cos(18)},{-6.35+1.05*sin(18)})
    arc[start angle=18,end angle=162,radius=1.05]
    --({\cx+1.3*cos(162)},{-6.35+1.3*sin(162)});
}
\foreach \a in {198,342}{
  \draw[prooforange,line width=1pt,dashed] ({12+1.3*cos(\a)},{-6.35+1.3*sin(\a)})
    --({12+.7*cos(\a)},{-6.35+.7*sin(\a)});
}
\draw[prooforange,line width=1pt,dashed] ({12+.7*cos(342)},{-6.35+.7*sin(342)})
  arc[start angle=342,end angle=198,radius=.7];
\draw[black!40,densely dotted] (12,-6.35) circle (.45);
\draw[proofarrow] (3.55,-6.35)--(5.4,-6.35);
\draw[proofarrow] (8.55,-6.35)--(10.4,-6.35);
\node[proofnote,anchor=north,text width=4.2cm] at (2,-7.95) {$24$ ending strands\\and two passages:\\$24+2\cdot2=28$ strands.};
\node[proofnote,anchor=north,text width=4.2cm] at (7,-7.95) {First passage bypassed.\\One passage remains:\\$24+2\cdot1=26$ strands.};
\node[proofnote,anchor=north,text width=4.2cm] at (12,-7.95) {Smaller disc avoids\\the earlier bypass.\\$24$ strands remain.};
\node[proofpanel] at (0,-10) {(c) Strand counts and crossing bounds};
\node[proofbox,text width=13.7cm] at (7.5,-11.3)
  {$h$ passages left: $r-1+2h$ radial strands. Two strands belong to the current passage.\\[4pt]
   A boundary bypass crosses at most
   $\left\lfloor\dfrac{r-3+2h}{2}\right\rfloor
    =\left\lfloor\dfrac{r-3}{2}\right\rfloor+h$ other strands.};
\node[proofnote,text width=14.4cm] at (7.5,-13.4)
  {$\displaystyle
    \sum_{v\in W}f(v)\le2\tau
    \quad\Longrightarrow\quad\text{nonbranch cost}\le c\tau,$\\[8pt]
   $\displaystyle
    2f_p\le c,\quad\sum_{p\in P}f_p\le\tau
    \quad\Longrightarrow\quad\text{branch cost}
       \le\frac{2r+c-4}{4}\tau.$};
\node[proofbox,draw=proofblue,text width=12cm] at (7.5,-15.1)
  {$\displaystyle\text{total added cost}\le
    \frac{2r+5c-4}{4}\tau
    \le\frac{2r+5c-4}{4}T_c=\Lambda(r,c).$};
\end{tikzpicture}
}%
\caption{Local smoothing in Lemma~\ref{lem:smoothing}. Panel (a) reconnects three passages at a nonbranch vertex. Panel (b) uses the graph of Figure~\ref{fig:routing}: two routes pass through $p$ and $24$ end there. Blue and dashed orange strands denote passages; grey strands end at $p$. Each illustrated bypass crosses ten grey strands, within the respective bounds $13$ and $12$, and the smaller disc is disjoint from the earlier bypass. Each nonbranch vertex is visited once, contributing zero additional crossings. Hence the cost is at most $25$, while \eqref{eq:K6average} gives $\crn(K_{6,8})\ge72$. Panel (c) gives the general estimate.}
\label{fig:smoothing}
\end{figure}

\FloatBarrier
\Needspace{15\baselineskip}
\subsection{Completing the near-range proof}

\begin{proposition}[Near range]\label{prop:near}
If $G$ is $r$-critical, $r\ge19$, and $n<5r/4$, then
$\crn(G)\ge\crn(K_r)$.
\end{proposition}
\begin{proof}
The cases $c\le5$ were settled above. For $c\ge6$, the strict range gives $r>4c$. The reserved side satisfies
\[
b=r-(c+1)-\floor{3(c+1)/2}
>\frac{3c-5}{2}\ge\frac{13}{2},
\]
so $b\ge7$ and the routing lemma applies.

Let $G'$ be the union of the immersion paths and $\mathcal R$ the reserved $K_{c,b}$. In a crossing-minimal good drawing of $G$, let $X$ count crossings between two $G'$ edges and $Y$ crossings between two $\mathcal R$ edges. Since the edge sets are disjoint, the crossings counted by $X$ and $Y$ are distinct. The remaining crossings, including those between the two edge sets, contribute a nonnegative amount; hence $\crn(G)\ge X+Y$. The smoothing lemma gives $X+\Lambda(r,c)\ge\crn(K_r)$, while restriction to $\mathcal R$ gives $Y\ge\crn(K_{c,b})$. Therefore
\begin{equation}
\crn(G)\ge\crn(K_r)-\Lambda(r,c)+\crn(K_{c,b}).
\label{eq:near-master}
\end{equation}
By Lemma~\ref{lem:bipartite-average}, it is enough to prove
\begin{equation}
W(r,c):=\frac{c(c-1)}5\floor{b/2}\floor{(b-1)/2}
\ge\Lambda(r,c).
\label{eq:near-sufficient}
\end{equation}

\paragraph{Excess at least fourteen.}
For an integer $b\ge2$, both $\lfloor b/2\rfloor$ and $\lfloor(b-1)/2\rfloor$ are at least $(b-2)/2$. Also
$b\ge r-\frac52(c+1)$, so $b-2\ge(2r-5c-9)/2>0$ in the present range. Therefore
\[
W(r,c)\ge\frac{c(c-1)}{80}(2r-5c-9)^2.
\]
Also $T_c\le c(c+1)/4$ gives
\[
\Lambda(r,c)\le\frac{c(c+1)(2r+5c-4)}{16}.
\]
Multiplying by $80/c>0$ reduces the desired comparison to
\[
F_c(r):=(c-1)(2r-5c-9)^2-5(c+1)(2r+5c-4)>0.
\]
Put $u=c-14\ge0$ and $v=r-4c>0$. Expanding gives the exact identity
\begin{align*}
F_c(r)={}&9u^3+250u^2+1798u+807\\
&+(12u^2+278u+1566)v+4(u+13)v^2.
\end{align*}
Every term is nonnegative and the constant term is positive. Thus the comparison holds for every $c\ge14$.

\paragraph{The smaller excesses.}
For $6\le c\le13$, the following thresholds give positive initial margins:
\[
\begin{array}{c|rrrrrrrr}
c&6&7&8&9&10&11&12&13\\\hline
R_c&30&34&37&40&42&46&49&53\\
W(R_c,c)-\Lambda(R_c,c)&1&63/10&269/5&401/10&9/2&209/4&1881/10&7353/20.
\end{array}
\]
The margins in the last row are obtained by substituting each displayed $R_c$ into the definitions of $b$, $W$, and $\Lambda$. These are exact rational calculations; no floating-point arithmetic or numerical optimization is used.

For an integer $b\ge2$,
\[
\floor{(b+1)/2}\floor{b/2}-\floor{b/2}\floor{(b-1)/2}
=\floor{b/2}.
\]
Consequently the increase of $W$ when $r$ increases by one is $c(c-1)\lfloor b/2\rfloor/5$. At the listed thresholds $b\ge13$, so for all larger $r$ this increase is at least
\[
6c(c-1)/5\ge c(c+1)/8\ge T_c/2.
\]
The increase of $\Lambda$ is exactly $T_c/2$. Hence every positive initial margin persists, settling all $r\ge R_c$.

It remains to treat $4c<r<R_c$. The table gives no such integer for $c=12,13$. For the other excesses the remaining integer values are exactly
\[
\begin{array}{c|rrrrrr}
c&6&7&8&9&10&11\\\hline
r&25\text{--}29&29\text{--}33&33\text{--}36&37\text{--}39&41&45.
\end{array}
\]
Thus there are 19 remaining pairs. Each satisfies $25\le r\le45$ and $r<5c$. We cover all of them by one sampling calculation on a rectangle. Since $\rho=1+c/r$, the inequalities $r<5c$ and $r>4c$ give the bounds
\[
6/5\le\rho:=(r+c)/r\le5/4,\qquad
1/45\le x:=1/r\le1/25.
\]
Gallai's edge bound applies because $r+2\le r+c\le2r-1$. After division by $r^2$ it gives
\[
\mu_G(\rho,x)=(-\rho^2+4\rho-2-\rho x-2x^2)/2.
\]
Since $n=r+c\ge31$, a twelve-vertex sample is admissible. Suppose a counterexample exists and let $C_{12}$ and $M_{12}$ be the crossing and edge counts in the inherited twelve-vertex sample. Then
\[
\mathbb E C_{12}\ge5\mathbb E M_{12}-225
\ge660\,\frac{r^2\mu_G(\rho,x)}{(n)_2}-225,
\qquad
\mathbb E C_{12}\le\frac{(12)_4}{(n)_4}r^4T(x).
\]
A positive difference between these lower and upper bounds gives a contradiction. Multiplying the difference by
$(n)_4/r^4=\rho(\rho-x)(\rho-2x)(\rho-3x)>0$ gives
\begin{align}
J(\rho,x)={}&660\mu_G(\rho,x)(\rho-2x)(\rho-3x)\notag\\
&-225\rho(\rho-x)(\rho-2x)(\rho-3x)
-11880T(x),
\label{eq:near-polynomial}
\end{align}
where $T$ is defined in \eqref{eq:smooth-target}. The constants are $5(12)_2=660$, $\lfloor203(12-2)/9\rfloor=225$, and $(12)_4=11880$.

On $[6/5,5/4]\times[1/45,1/25]$, the polynomial has degree at most four in each variable, giving tensor degree bound $(4,4)$. The verifier computes its 25 Bernstein coefficients and reconstructs it; their least value is $433544/91125>0$. Lemma~\ref{lem:bernstein-positive} therefore gives $J>0$ throughout that rectangle, settling every remaining pair by the same sampling contradiction. Together with \eqref{eq:near-master}, this proves the proposition.
\end{proof}

\FloatBarrier
\section{The remaining orders}\label{sec:middle}

We now exclude counterexamples with $n\ge5r/4$ and $r\ge19$. For $n\le3r$, the edge bounds and maximum-clique sampling give polynomial inequalities on a bounded parameter domain. For $n\ge3r$, one ordinary sampling rule suffices. Both checks use the positive-basis criterion proved in Appendix~\ref{app:verification}.

\subsection{Two sampling tests}

Fix a crossing-minimal good drawing of a counterexample, and let $M$ be a proved lower bound on its edge count. This lower bound may be real; the averaging argument uses only $m\ge M$. In a uniform $s$-vertex sample, an edge survives with probability $(s)_2/(n)_2$, and a crossing with probability $(s)_4/(n)_4$. Consequently \eqref{eq:BK5-int} gives a contradiction whenever
\begin{equation}
5M\frac{(s)_2}{(n)_2}-\floor{203(s-2)/9}
-\frac{(s)_4}{(n)_4}\bigl(\overline Z(r)-1\bigr)>0,
\qquad 4\le s\le n.
\label{eq:ordinary-test}
\end{equation}
Indeed, the first two terms lower-bound the expected number of crossings in the inherited sample drawing, while the last term is an upper bound under the counterexample assumption.

The second test uses crossings already forced by a maximum clique to sharpen the upper bound on the expected sampled crossing count. We first bound sampled edges, then sampled crossings. Fix a maximum clique $Q$ of order $w$, and put $t=n-w$. Choose exactly $u$ vertices outside $Q$ and exactly $v$ inside $Q$, uniformly and independently between the two parts. We require
\[
4\le u\le t,\qquad4\le v\le w,\qquad w\ge7.
\]
For $w<7$ we use the ordinary test. For an admissible clique sample, the outside--outside and cross-part edge probabilities are
\begin{equation}
A=\frac{(u)_2}{(t)_2},\qquad B=\frac{uv}{wt}.
\label{eq:AB}
\end{equation}
We use only choices with $A\ge B$ and $2B\ge A$. These restrictions allow the two available edge counts to be combined with nonnegative coefficients. If $a$ counts outside--outside edges and $b'$ counts cross-part edges, then the edge total is $a+b'+\binom w2$, and the sum of degrees outside $Q$ is $2a+b'$. Hence
\[
a+b'\ge M-\binom w2,\qquad2a+b'\ge t(r-1).
\]
The identity $Aa+Bb'=(A-B)(2a+b')+(2B-A)(a+b')$ has nonnegative coefficients. The selected clique vertices contribute exactly $\binom v2$ edges. Adding their contribution to $Aa+Bb'$ shows that the expected sampled edge count is at least
\begin{equation}
E_0=\binom v2+(A-B)t(r-1)+(2B-A)\left(M-\binom w2\right).
\label{eq:E0}
\end{equation}

A crossing with $j$ endpoints outside $Q$ survives with probability
\begin{equation}
P_j=\frac{(u)_j(v)_{4-j}}{(t)_j(w)_{4-j}},\qquad
P=\max_{0\le j\le4}P_j.
\label{eq:Pj}
\end{equation}
Let $C$ be the drawing's total crossing count and $C_Q$ its count between two clique edges. These are precisely the crossings with all four endpoints in $Q$, since $Q$ is complete. Restriction to $Q$ and \eqref{eq:Cq} give $C_Q\ge\crn(K_w)\ge C_w$. The clique crossings survive with probability $P_0$, and every other crossing survives with probability at most $P$. The expected sampled crossing count is at most
\[
P(C-C_Q)+P_0C_Q=PC-(P-P_0)C_Q
\le P\bigl(\overline Z(r)-1\bigr)-(P-P_0)C_w,
\]
because $P-P_0\ge0$. The correction subtracts the excess weight assigned to clique crossings by the uniform upper probability $P$. Combining this upper bound with \eqref{eq:BK5-int} proves the following test.

\begin{lemma}[Maximum-clique sample]\label{lem:max-clique-test}
Under the stated sampling and coefficient conditions, a counterexample is impossible if
\begin{equation}
5E_0-\floor{203(u+v-2)/9}
-P\bigl(\overline Z(r)-1\bigr)+(P-P_0)C_w>0.
\label{eq:clique-test}
\end{equation}
\end{lemma}

The sequence $(P_j)$ is log-concave. Its successive ratios are
\[
R_j:=\frac{P_{j+1}}{P_j}
=\frac{u-j}{t-j}\frac{w-3+j}{v-3+j}
\qquad(0\le j\le3).
\]
Both factors are positive and nonincreasing in $j$. Indeed, the differences of successive factors are
\[
\frac{t-u}{(t-j)(t-j-1)}\ge0,\qquad
\frac{w-v}{(v-3+j)(v-2+j)}\ge0
\qquad(0\le j\le2).
\]
Thus $R_0\ge R_1\ge R_2\ge R_3$, or equivalently $P_j^2\ge P_{j-1}P_{j+1}$ for $1\le j\le3$: the sequence $(P_j)$ is log-concave. Moreover, $A\ge B$ implies
\[
\frac{u-1}{t-1}\ge\frac vw\ge\frac{v-2}{w-2},
\]
where the last inequality follows from $v\le w$. Hence $R_1\ge1$, and $R_0\ge R_1$ gives
\begin{equation}
P_2\ge P_1\ge P_0,\qquad P=\max\{P_2,P_3,P_4\}.
\label{eq:three-modes}
\end{equation}
Thus only indices $2,3,4$ are needed in the sampling test.

Figure~\ref{fig:sampling} gives a sampling example with a positive clique correction.
\begin{figure}[!htbp]
\centering
\scalebox{0.88}{%
\begin{tikzpicture}[x=1cm,y=1cm]
\node[proofpanel] at (0,0) {(a) A sample of $K_8\mathbin{\dot\cup}P_9$: $v=4$, $u=5$};
\draw[rounded corners,fill=proofblue!3,draw=proofblue] (0,-.5) rectangle (6.5,-4.4);
\draw[rounded corners,fill=prooforange!3,draw=prooforange] (8.5,-.5) rectangle (15,-4.4);
\node[proofnote] at (3.25,-.9) {Maximum clique $Q$, $|Q|=w$};
\node[proofnote] at (11.75,-.9) {Outside $Q$, $t=n-w$ vertices};
\foreach \i/\qx/\qy in {1/-713/697,2/-1003/17,3/-689/-733,4/19/-1007,5/727/-683,6/991/23,7/691/729,8/-31/997}{
  \coordinate (q\i) at ({3.25+.00115*\qx},{-2.65+.00115*\qy});
}
\foreach \i in {1,...,8}{\foreach \j in {1,...,8}{
  \ifnum\i<\j\draw[black!20,line width=.35pt] (q\i)--(q\j);\fi
}}
\foreach \i in {1,...,8}{
  \ifnum\i<5\node[proofbranch] at (q\i) {};
  \else\node[proofvertex,draw=proofblue,minimum size=6pt] at (q\i) {};\fi
}
\foreach \x/\y/\i in {10.2/-1.85/1,11.75/-1.85/2,13.3/-1.85/3,13.3/-2.6/4,11.75/-2.6/5,10.2/-2.6/6,10.2/-3.35/7,11.75/-3.35/8,13.3/-3.35/9}{
  \coordinate (o\i) at (\x,\y);
}
\foreach \a/\b in {1/2,2/3,3/4,4/5,5/6,6/7,7/8,8/9}{\draw[black!45] (o\a)--(o\b);}
\foreach \i in {1,...,9}{
  \ifnum\i<6\node[proofoutside,fill=prooforange] at (o\i) {};
  \else\node[proofoutside] at (o\i) {};\fi
}
\node[proofnote] at (3.25,-4.05) {$v$ selected clique vertices.};
\node[proofnote] at (11.75,-4.05) {$u$ selected outside vertices.};
\node[proofnote,text width=14cm] at (7.5,-5)
  {Circles: clique vertices. Squares: outside vertices. Filled symbols: selected.\\
   Here $w=8$, $t=9$, $u=5$, $v=4$, so $A=B=5/18$.};
\node[proofpanel] at (0,-5.9) {(b) Crossing endpoint types};
\foreach \j in {0,...,4}{
  \begin{scope}[shift={({1.45+3*\j},-7.25)}]
  \draw[black!70] (-.7,.55)--(.7,-.55);
  \draw[black!70] (-.7,-.55)--(.7,.55);
  \foreach \x/\y/\i in {-.7/.55/1,.7/.55/2,-.7/-.55/3,.7/-.55/4}{
    \ifnum\i>\j\node[proofvertex,draw=proofblue,minimum size=6pt] at (\x,\y) {};
    \else\node[proofoutside,fill=prooforange] at (\x,\y) {};\fi
  }
  \node[proofnote] at (0,-1.05) {$j=\j$};
  \node[proofnote] at (0,-1.6) {$P_{\j}$};
  \end{scope}
}
\draw[decorate,decoration={brace,mirror,amplitude=5pt},proofblue]
  (6.5,-9.35)--(14.4,-9.35);
\node[proofnote] at (10.45,-9.8) {Potential maximizing indices};
\node[proofnote,text width=5.4cm] at (2.75,-9.65) {$P_2\ge P_1\ge P_0$\\under the sampling conditions};
\node[proofnote] at (7.5,-10.55)
  {$P_j=\dfrac{(u)_j(v)_{4-j}}{(t)_j(w)_{4-j}},\qquad
   P=\max\{P_2,P_3,P_4\}.$};
\node[proofpanel] at (0,-11.45) {(c) Clique-crossing correction};
\node[proofbox,text width=5.7cm] (cq) at (3.35,-12.55)
  {Two clique edges cross:\\$C_Q$ crossings, each survives with $P_0$.};
\node[proofbox,text width=6.5cm] (rest) at (11.15,-12.55)
  {All remaining crossings:\\$C-C_Q$ crossings, each survives with at most $P$.};
\node[proofbox,text width=13.6cm,draw=proofblue] at (7.5,-14.25)
  {$\mathbb E[\text{sampled crossings}]
       \le P(C-C_Q)+P_0C_Q=PC-(P-P_0)C_Q$\\[5pt]
   $C_Q\ge C_w,\ P\ge P_0
       \quad\Longrightarrow\quad
       \mathbb E[\text{sampled crossings}]
       \le PC-(P-P_0)C_w.$};
\node[proofnote,text width=14cm] at (7.5,-15.9)
  {For the example, $P_0=1/70$ and $P=5/84$. Since $C_8=21/2$,\\
   the general correction is $(P-P_0)C_8=19/40>0$.};
\end{tikzpicture}
}%
\caption{Sampling on $K_8\mathbin{\dot\cup}P_9$ with $v=4$ and $u=5$. Panel (a) illustrates the sampling distribution independently of the counterexample assumption. The convex clique drawing has $\binom84=70$ distinct crossings, and every four-vertex clique sample has one crossing, agreeing with $70P_0=1$. Panel (b) lists the endpoint types for a general graph. Panel (c) gives the crossing upper bound before substituting the counterexample bound on $C$. Here $P_0=1/70$, $P=5/84$, and $(P-P_0)C_8=19/40>0$.}
\label{fig:sampling}
\end{figure}

\FloatBarrier
\subsection{The bounded order range}

\paragraph{Parameters and graph points.} Put $\rho=n/r$, $b=w/r$, $x=1/r$, and $d=3-\rho$. In this section, $b$ denotes the normalized clique size; in Section~\ref{sec:near}, it counts reserved singletons. Since $n>r$, an $r$-clique would be a proper subgraph of $G$, contradicting criticality. Thus $2\le w\le r-1$, and every graph point lies in
\begin{equation}
\mathcal B=[5/4,3]\times[0,1]\times[0,1/19].
\label{eq:bounded-domain}
\end{equation}
Throughout this section, $b=w/r$ denotes normalized clique size. We call a triple $(n/r,w/r,1/r)$ arising from the integer graph parameters a \emph{graph point}. Polynomial positivity will be established on closed boxes covering $\mathcal B$, including the boundary $x=0$. This box also contains triples that do not arise from graphs. The probabilistic inequalities are applied only at graph points, where $x>0$ and their denominators are positive. Once those denominators have been cleared, it is sufficient to certify the resulting polynomials on the larger closed boxes. No probabilistic interpretation is needed at $x=0$.

\paragraph{The edge lower bound.} The following polynomials are lower bounds on $m/r^2$ whenever their displayed conditions hold:
\begin{align}
\mu_D&=\rho(1-x)/2,\notag\\
\mu_G&=(-\rho^2+4\rho-2-\rho x-2x^2)/2,\notag\\
\mu_2&=(3-d^2+2bd-2b^2)/2-3x,\notag\\
\mu_3&=(3+d^2-6bd+6b^2)/2-3x\quad(D\ge2w),\notag\\
\mu_R&=(3-bd)/2-(3/2+2b-3d/2)x-3x^2/2\quad(D\le w).
\label{eq:edge-primitives}
\end{align}
Minimum degree gives $\mu_D$, and Theorem~\ref{thm:clique-edge} gives the last three. For $\rho<2$, integrality implies $n\le2r-1$, while $n-r\ge r/4\ge19/4>2$ gives $n\ge r+2$. Thus Gallai's bound gives $\mu_G$. For $\rho\ge2$,
\[
\mu_G-\mu_D=-\frac{(\rho-1)(\rho-2)}2-x^2\le0,
\]
so $\mu_G$ remains a valid lower bound throughout the domain. A certificate record consists of one of these five lower bounds, its \emph{edge primitive}, together with an ordinary or clique sample. The \emph{guards} are the conditions ensuring that the primitive and sample are admissible at every graph point in the record's box. A record uses just one valid primitive. It suffices that this choice makes its sampling test positive; the verifier does not need to determine which of the five edge estimates is strongest at each point.

\paragraph{Ordinary samples.} Let $\mu$ be that primitive. Substitute $M=r^2\mu$ in \eqref{eq:ordinary-test} and multiply by $(n)_4/r^4=\rho(\rho-x)(\rho-2x)(\rho-3x)>0$ at every admissible graph point. This gives the ordinary polynomial
\begin{align}
F_s={}&5(s)_2\mu(\rho-2x)(\rho-3x)
-\floor{203(s-2)/9}\rho(\rho-x)(\rho-2x)(\rho-3x)\notag\\
&-(s)_4T(x).
\label{eq:ordinary-polynomial}
\end{align}
Its positivity excludes the counterexample by an ordinary sample.

\paragraph{Clique samples.} The test \eqref{eq:clique-test} combines a lower bound on sampled edges with an upper bound on surviving crossings. After normalization it is a rational inequality in $(\rho,b,x)$. We clear its positive denominators to obtain a polynomial whose sign can be certified on closed boxes. The numerators and denominator factors defined below express the test as a polynomial with exact rational coefficients.

Put $g=\rho-b$, and define
\begin{align*}
\Delta&=bg(g-x),&B_4&=b(b-x)(b-2x)(b-3x),\\
f_1&=(u+v-1)b-v\rho+vx,&f_2&=2v\rho-(u+2v-1)b-2vx,\\
h&=80(b-6x),&q&=b-7x.
\end{align*}
Since $w=b/x$ and $t=g/x$, write $A=x^2\widehat A$ and $B=x^2\widehat B$, where
\[
\widehat A=\frac{u(u-1)}{g(g-x)},\qquad
\widehat B=\frac{uv}{bg}.
\]
With common denominator $\Delta$, these expressions satisfy
\[
\widehat A-\widehat B=uf_1/\Delta,\qquad
2\widehat B-\widehat A=uf_2/\Delta.
\]
The sample capacities give $b\ge7x$, $g\ge ux\ge4x$, and $x>0$ at a graph point, so $\Delta>0$. Thus the coefficient conditions are exactly $f_1,f_2\ge0$. Moreover,
\[
t(r-1)=\frac{g(1-x)}{x^2},\qquad
M-\binom w2=\frac{\mu-b(b-x)/2}{x^2}.
\]
Substitution in \eqref{eq:E0} cancels the factors $x^2$. In the following display, $E$ is a polynomial numerator, not the extra-deletion count of Section~\ref{sec:compression}: 
\[
E=\binom v2\Delta+uf_1g(1-x)
+uf_2\left(\mu-\frac{b(b-x)}2\right),
\qquad L=5E-\floor{203(u+v-2)/9}\Delta.
\]
Thus $E_0=E/\Delta$ and $L/\Delta=5E_0-\lfloor203(u+v-2)/9\rfloor$.

The lower bound $C_w$ defined in \eqref{eq:Cq} satisfies the exact identity $C_w/r^4=c_5:=B_4q/h$. Write $P_j=x^4\widehat P_j$. If $j$ attains the maximum, \eqref{eq:clique-test} becomes
\[
\frac L\Delta-\widehat P_jT(x)+(\widehat P_j-\widehat P_0)c_5>0.
\]
The coefficient $\widehat P_j-\widehat P_0$ is nonnegative, as required when the graph-theoretic clique lower bound was substituted. Rearranging gives
\begin{equation}
L/\Delta-\widehat P_0c_5-\widehat P_j\bigl(T(x)-c_5\bigr)>0.
\label{eq:normalized-clique-test}
\end{equation}
This substitution does not require a sign assumption on $T(x)-c_5$.

The three normalized probabilities are
\[
\widehat P_2=\frac{(u)_2(v)_2}{g(g-x)b(b-x)},\quad
\widehat P_3=\frac{(u)_3v}{g(g-x)(g-2x)b},\quad
\widehat P_4=\frac{(u)_4}{g(g-x)(g-2x)(g-3x)}.
\]
Define the remaining denominator factors by
\[
\begin{array}{c|cc}
j&D_j&\delta_j\\\hline
2&b-x&1\\
3&g-2x&1\\
4&(g-2x)(g-3x)&b.
\end{array}
\]
Set $N_j=(u)_j(v)_{4-j}$. Formula~\eqref{eq:Pj} gives
$\Delta D_j\widehat P_j=N_j\delta_j$.
Also $\widehat P_0=(v)_4/B_4$, so $\widehat P_0c_5=(v)_4q/h$. The sample capacities ensure $b-6x>0$, $b-x>0$, and $g-3x>0$; hence $h\Delta D_j>0$ at every admissible graph point. The quantities in the reduction have the following roles.

\begin{center}
\small
\begin{tabular}{@{}>{\raggedright\arraybackslash}p{.17\linewidth}>{\raggedright\arraybackslash}p{.23\linewidth}>{\raggedright\arraybackslash}p{.53\linewidth}@{}}
\toprule
Contribution & Quantities & Relation used\\
\midrule
Edge estimate & $f_1,f_2,E,L$ & Nonnegative edge weights give $E_0=E/\Delta$ and the crossing lower bound $L/\Delta$.\\[4pt]
Crossing estimate & $\widehat P_j,B_4,h,q$ & The normalized endpoint probabilities combine with the clique bound $c_5=B_4q/h$.\\[4pt]
Denominators & $g,\Delta,D_j,N_j,\delta_j$ & $\Delta D_j\widehat P_j=N_j\delta_j$ and $h\Delta D_j>0$ yield the polynomial $\Pi_j$.\\
\bottomrule
\end{tabular}
\end{center}

Multiplying \eqref{eq:normalized-clique-test} by $h\Delta D_j$ therefore gives the polynomial
\begin{equation}
\Pi_j=hD_jL-(v)_4q\Delta D_j
-N_j\delta_j\bigl(hT(x)-B_4q\bigr).
\label{eq:clique-polynomial}
\end{equation}
The ordinary polynomial has total degree at most four, and each clique polynomial has total degree at most six. The sample sizes $s$ or $(u,v)$ are fixed within a record, so their falling factorials and rounded intercepts are constants in these polynomials. The certificate checks $F_s>0$, or $\Pi_j>0$ for every index that can attain the maximum on its box. Appendix~\ref{app:verification} gives the sampling guards, the index guards, and the exact coefficient conversion.

\paragraph{A certificate record.}
The record \texttt{AAaBBB C 28 10 R} in \texttt{bounded.cert} assigns the clique sample $u=28$, $v=10$ and the edge primitive $\mu_R$ to the box
\[
\mathcal Q=
\left[\frac{41}{16},\frac{89}{32}\right]
\times\left[\frac78,1\right]
\times\left[0,\frac1{19}\right].
\]
The path can be read directly from the root box. Its letters \texttt{AAa} successively give the $\rho$ intervals
\[
[5/4,3]\longrightarrow[17/8,3]\longrightarrow[41/16,3]
\longrightarrow[41/16,89/32].
\]
Its letters \texttt{BBB} give $[0,1]\to[1/2,1]\to[3/4,1]\to[7/8,1]$ in the $b$ coordinate; the $x$ interval is unchanged. For instance, the critical graph $K_{17}\vee C_{35}$ has parameter point $(\rho,b,x)=(13/5,19/20,1/20)$ in this box.

Throughout $\mathcal Q$,
\[
d-b\le-\frac7{16},\qquad
g-28x\ge\frac{27}{304},\qquad
b-10x\ge\frac{53}{152}.
\]
Thus $\mu_R$ is available, both sample capacities hold, and the denominator factors are positive. The edge coefficients satisfy
\[
f_1=37b-10\rho+10x\ge\frac{73}{16},\qquad
f_2=20\rho-47b-20x\ge\frac{243}{76}.
\]
Moreover,
\[
\frac{\widehat P_4}{\widehat P_3}-1
=\frac{35b-10\rho+30x}{10(g-3x)}>0,
\qquad 35b-10\rho+30x\ge\frac{45}{16}.
\]
By log-concavity and \eqref{eq:three-modes}, index $4$ attains the maximum at every graph point of the box. This record therefore requires only $\Pi_4>0$.

With the preceding definitions, its polynomial is
\begin{align*}
L={}&5\left(45\Delta+28f_1g(1-x)
+28f_2\left(\mu_R-\frac{b(b-x)}2\right)\right)-812\Delta,\\
\Pi_4={}&h(g-2x)(g-3x)L
-5040q\Delta(g-2x)(g-3x)\\
&-491400b\bigl(hT(x)-B_4q\bigr).
\end{align*}
In affine unit coordinates
\[
(t_1,t_2,t_3)=\left(\frac{32\rho-82}{7},\,8b-7,\,19x\right),
\]
the tensor degree bound is $(4,6,6)$. The 245 exact Bernstein coefficients have least value
\[
\frac{900929925}{131072}>0,
\]
attained at index $(0,0,0)$, where it equals $\Pi_4(41/16,7/8,0)$. Lemma~\ref{lem:bernstein-positive} gives strict positivity throughout $\mathcal Q$. Together with the guards, this excludes a counterexample whose graph point lies in that box.

\begin{proposition}[Bounded order range]\label{prop:bounded}
If $G$ is $r$-critical, $r\ge19$, and $5r/4\le n\le3r$, then $\crn(G)\ge\crn(K_r)$.
\end{proposition}
\begin{proof}
The file \texttt{bounded.cert} describes a binary partition of $\mathcal B$ into 463 closed boxes. The supplied verifier reconstructs every box from its path and checks its admissibility guards and positive Bernstein coefficients, independently of how the certificate was generated. There are 135 ordinary witnesses and 328 clique witnesses. Across them, it constructs 624 polynomials, computes all 110799 tensor Bernstein coefficients, and reconstructs every polynomial independently. The least coefficient is
\[
\frac{7704424178956295}{10634764226658304}>0.
\]
Lemma~\ref{lem:bernstein-positive} proves positivity throughout every box. The graph point of a counterexample belongs to a box whose witness satisfies all required guards. For a clique witness, the tested indices include a maximizer of the crossing-survival probability at that point. Positivity of the corresponding polynomial contradicts \eqref{eq:ordinary-test} or \eqref{eq:clique-test}.
\end{proof}

\subsection{The large order range}

\begin{proposition}[Large orders]\label{prop:large}
If $G$ is $r$-critical, $r\ge19$, and $n\ge3r$, then $\crn(G)\ge\crn(K_r)$.
\end{proposition}
\begin{proof}
Use minimum degree and the ordinary sample size $s=\lceil12n/r\rceil$. Here $n/r$ is unbounded, so we use its reciprocal $y=r/n$ as the order parameter. Choosing the sample size proportional to $n/r$ keeps $sr/n$ in a bounded interval; a further parameter records the effect of rounding. Put
\[
y=r/n\in(0,1/3],\quad x=1/r\in(0,1/19],\quad
\theta=s-12n/r\in[0,1],\quad z=sy=12+\theta y.
\]
The ceiling definition actually gives $0\le\theta<1$. We include $\theta=1$ only to certify positivity over the larger closed parameter box.
The sample is admissible: $s\ge36$, and
$n-12n/r=(r-12)n/r\ge21$ gives $s\le n$.

Since $n=1/(xy)$ and $s=z/y$, substituting in $(r-1)n(s)_2/(2(n)_2)$ gives expected sampled edge count at least
\[
\frac{(1-x)z(z-y)}{2y(1-xy)}.
\]
Similarly, substituting in $(s)_4r^4T(x)/(n)_4$ gives expected sampled crossing count at most
\[
\frac{z(z-y)(z-2y)(z-3y)T(x)}{(1-xy)(1-2xy)(1-3xy)}.
\]
Apply the unrounded bound \eqref{eq:BK5}; its intercept is $203(z-2y)/(9y)$. Using this valid, slightly weaker bound avoids a floor depending on the varying sample size, so one polynomial can treat every large order. Subtracting the crossing upper bound from the resulting lower bound and multiplying by the positive factor
$y(1-xy)(1-2xy)(1-3xy)$ gives
\begin{align}
H(y,x,\theta)={}&\frac52(1-x)z(z-y)(1-2xy)(1-3xy)\notag\\
&-\frac{203}{9}(z-2y)(1-xy)(1-2xy)(1-3xy)\notag\\
&-yz(z-y)(z-2y)(z-3y)T(x).
\label{eq:large-polynomial}
\end{align}
After replacing $z$ by $12+\theta y$, this is a polynomial in $(y,x,\theta)$, with degree at most $5,4,4$ in the respective variables. We check it on $[0,1/3]\times[0,1/19]\times[0,1]$, including the limiting face $y=0$. The verifier computes all 150 Bernstein coefficients and reconstructs the polynomial; their least value is $749/216>0$. By Lemma~\ref{lem:bernstein-positive}, $H$ is positive on the whole box, including the boundary $y=0$. In particular, $H>0$ at every graph point, giving the required sampling contradiction.
\end{proof}

\section{Proof of the conjecture}\label{sec:final}
\begin{proof}[Proof of Theorem~\ref{thm:main}]
Let $s=\chi(G)\ge r$, and choose an $s$-critical subgraph $H\subseteq G$. The critical case for $s\le18$ follows from Ackerman~\cite[Section~3.1]{Ackerman}. For $s\ge19$, Proposition~\ref{prop:near} excludes a critical counterexample with $|H|<5s/4$, Proposition~\ref{prop:bounded} excludes one with $5s/4\le|H|\le3s$, and Proposition~\ref{prop:large} excludes every larger order. Thus $\crn(H)\ge\crn(K_s)$ in all cases. Subgraph monotonicity now gives
\[
\crn(G)\ge\crn(H)\ge\crn(K_s)\ge\crn(K_r),
\]
where the last inequality uses $K_r\subseteq K_s$.
\end{proof}

\appendix
\section{Exact certificates and reproducibility}\label{app:verification}

The computational verification concerns positivity of the sampling polynomials: one bivariate polynomial for the residual near range, a partition certificate for the bounded order range, and one polynomial on a fixed box for the large order range. Each is verified by exact rational Bernstein coefficients.

\subsection{Positive coefficients on a box}

For a coordinate interval $[a,b]$ with $a<b$, use $t=(x-a)/(b-a)$ to map it to $[0,1]$. If a polynomial has degree at most $d_j$ in coordinate $j$, its \emph{tensor degree bound} is $(d_1,\ldots,d_k)$ and its tensor Bernstein representation has $\prod_j(d_j+1)$ coefficients. These are bounds in the individual variables, rather than a total-degree bound: $(4,4)$ gives $5\cdot5=25$ coefficients, and $(5,4,4)$ gives $6\cdot5\cdot5=150$. For $0\le t\le1$, put $B_i^d(t)=\binom di t^i(1-t)^{d-i}$. These functions are nonnegative and sum to one.

\begin{lemma}[Positive-basis criterion]\label{lem:bernstein-positive}
Let $f$ be a polynomial on a $k$-dimensional closed box. In affine unit-interval coordinates, suppose
\[
f(t_1,\ldots,t_k)=\sum_{0\le i_j\le d_j}c_{i_1,\ldots,i_k}
\prod_{j=1}^k B_{i_j}^{d_j}(t_j).
\]
Then $f\ge\min c_{i_1,\ldots,i_k}$ throughout the box.
\end{lemma}
\begin{proof}
The product basis functions are nonnegative, and their sum is
$\prod_j\sum_{i_j}B_{i_j}^{d_j}(t_j)=1$.
Thus $f$ is a weighted average of its coefficients.
\end{proof}

The conversion uses only rational arithmetic. If $a_{p_1,\ldots,p_k}$ are the power coefficients after affine substitution, then
\begin{equation}
c_{i_1,\ldots,i_k}=\sum_{p_j\le i_j}a_{p_1,\ldots,p_k}
\prod_{j=1}^k\frac{\binom{i_j}{p_j}}{\binom{d_j}{p_j}}.
\label{eq:bernstein-conversion}
\end{equation}
Indeed,
\[
t^p=\sum_{i=p}^d\frac{\binom ip}{\binom dp}B_i^d(t).
\]
The identity $\binom ip\binom di=\binom dp\binom{d-p}{i-p}$ reduces its right side to
$t^p\sum_{j=0}^{d-p}\binom{d-p}j t^j(1-t)^{d-p-j}=t^p$.
Applying this identity in every coordinate proves \eqref{eq:bernstein-conversion}. The verifier includes every tensor coefficient, including a zero coefficient if present. It also expands the resulting Bernstein representation back into the power basis and compares it with the original polynomial after affine substitution.

\subsection{Guards for the bounded order certificate}

Write a box as $[\rho_0,\rho_1]\times[b_0,b_1]\times[x_0,x_1]$.
The guards below are sufficient conditions on a whole box; they need not characterize every admissible sample. The primitives $\mu_D,\mu_G,\mu_2$ are always available at graph points. The verifier includes $\mu_R$ if
\[
3-\rho_0-b_0\le0\quad\text{or}\quad3-\rho_0-b_0<x_0.
\]
In the first case $D\le w$ throughout the box. In the second case an actual positive integer $D-w$ would give $d-b\ge x\ge x_0$, contradicting the upper bound. Thus $D\le w$ at every actual graph point in either case. Similarly, it includes $\mu_3$ if
\[
3-\rho_1-2b_1\ge0\quad\text{or}\quad3-\rho_1-2b_1>-x_0.
\]
A negative integer $D-2w$ would give $d-2b\le-x\le-x_0$, contradicting the second guard. These arguments justify the selected edge primitive at graph points; the polynomial sign test itself holds on the entire box.

An ordinary record requires $s\ge4$ and $sx_1\le\rho_0$, which guarantees $s\le n$. A clique record requires $u,v\ge4$ and
\[
b_0\ge\max\{7,v\}x_1,\qquad \rho_0-b_1\ge ux_1,
\qquad b_0>0,\quad \rho_0-b_1>x_1.
\]
These inequalities give the sample capacities and positive denominators. The affine coefficient numerators have minima
\[
\min f_1=(u+v-1)b_0-v\rho_1+vx_0,\qquad
\min f_2=2v\rho_0-(u+2v-1)b_1-2vx_1.
\]
Both must be nonnegative before a clique test is used. These guards justify the sampling inequality and the signs of its coefficients. The separate Bernstein check then proves the strict positivity needed for a contradiction.

To locate the maximum in \eqref{eq:three-modes}, the signs of $R_2-1$ and $R_3-1$ are the signs of
\[
U_2=(u+v-3)b-(v-1)\rho+(2v-u)x,\qquad
U_3=(u+v-3)b-v\rho+3vx.
\]
Indeed, subtracting one from the ratio formulas gives $R_2-1=U_2/((g-2x)(v-1))$ and $R_3-1=U_3/((g-3x)v)$, whose denominators are positive. Since $R_2\ge R_3$, if $\max U_2\le0$, index $2$ attains the maximum everywhere; if $\min U_3\ge0$, index $4$ does. Otherwise the verifier checks index $3$, includes index $2$ when $\min U_2<0$, and includes index $4$ when $\max U_3>0$. These choices include a maximizing index at every point: $U_2<0$ gives a maximum at index $2$, $U_3>0$ gives one at index $4$, and $U_2\ge0\ge U_3$ gives one at index $3$, by the signs of $R_2-1$ and $R_3-1$. When $U_2=0$ or $U_3=0$, the adjacent indices tie, so checking either tied index suffices. The exact affine minima and maxima are found by selecting the appropriate endpoint for each coefficient, with the sign of $2v-u$ determining the $x$ endpoint in $U_2$.

\subsection{Records and independent coverage}

Each line of \texttt{bounded.cert} has five fields:
\[
\texttt{path}\quad\texttt{kind}\quad p\quad q\quad\texttt{primitive}.
\]
Starting from $\mathcal B$, each letter of the path bisects one coordinate interval at its exact rational midpoint. The letters \texttt{a,A} select the lower and upper $\rho$ halves, \texttt{b,B} the $b$ halves, and \texttt{x,X} the $x$ halves. Thus the path determines the box independently of the generator.

For an \texttt{F} record, $p=s\ge4$, $q=0$, and the primitive is one of \texttt{D,G,Q2,Q3,R}, corresponding to $\mu_D,\mu_G,\mu_2,\mu_3,\mu_R$. The verifier checks its guards and all coefficients of \eqref{eq:ordinary-polynomial}. For a \texttt{C} record, $(p,q)=(u,v)$ with $u,v\ge4$; it checks the selected primitive, the sampling guards, and \eqref{eq:clique-polynomial} for each required index. No other record kind is accepted.

The verifier checks the certificate independently of its generator. It builds the prefix tree of the paths and rejects duplicates, a path extending a leaf, or a leaf extending another path. Every internal node must have exactly one complete child pair: \texttt{a,A}, \texttt{b,B}, or \texttt{x,X}. Each leaf has exactly one record. Two closed halves cover their parent interval, including the midpoint. Induction up this complete finite tree therefore proves coverage of $\mathcal B$. No assumption about the certificate generator enters that argument.

\subsection{Execution}

Verification uses four files: \texttt{polynomials.py} implements rational polynomials and the coefficient conversion; \texttt{sampling.py} constructs the displayed inequalities and guards; \texttt{bounded.cert} supplies the bounded-range witnesses; and \texttt{verify.py} checks every witness, the complete tree, and the near- and large-order polynomials. Python~3.10 or later suffices, using only the standard library. No compiler, solver, network access, or third-party package is needed.

The verification files are available through Zenodo~\cite{VerificationArchive}. After downloading and extracting the archive, enter \texttt{repo/} and run \texttt{python -B -u verify.py}. This writes \texttt{verification-results.json}, computes the bounded-range counts stated in Proposition~\ref{prop:bounded}, checks the 25 near-range and 150 large-order coefficients, and finishes with \texttt{ALL CHECKS PASSED}. The accompanying \texttt{README.md} gives the expected results and identifies the functions that implement each formula. Exact rational arithmetic, polynomial reconstruction, and the complete-tree check verify the inequalities and domain coverage from the supplied certificate. The certificate generator is not needed.

\paragraph{Scope of the computation.}
The program verifies the stated polynomial inequalities, the conditions for applying them, and coverage of the parameter domain. Its checks are confined to these numerical statements. The graph-theoretic reductions, routing construction, and smoothing argument are proved in the preceding sections.

\clearpage
\section*{Use of AI-assisted tools}
ChatGPT (OpenAI) was used for literature search, language editing, algebraic checks, diagram preparation, and development of the computer ancillary files. The author is responsible for the mathematical arguments, computations, references, and final text.


\begin{thebibliography}{99}
\small

\bibitem{OporowskiZhao}
B.~Oporowski and D.~Zhao,
\newblock Coloring graphs with crossings,
\newblock \emph{Discrete Math.} 309 (2009), 2948--2951.

\bibitem{HararyHill}
F.~Harary and A.~Hill,
\newblock On the number of crossings in a complete graph,
\newblock \emph{Proc. Edinburgh Math. Soc.} 13 (1963), 333--338.

\bibitem{ACF}
M.~O. Albertson, D.~W. Cranston and J.~Fox,
\newblock Crossings, colorings, and cliques,
\newblock \emph{Electron. J. Combin.} 16 (2009), R45.

\bibitem{Ackerman}
E.~Ackerman,
\newblock On topological graphs with at most four crossings per edge,
\newblock \emph{Comput. Geom.} 85 (2019), 101574.

\bibitem{BaratToth}
J.~Bar\'at and G.~T\'oth,
\newblock Towards the Albertson conjecture,
\newblock \emph{Electron. J. Combin.} 17 (2010), R73.

\bibitem{BK}
A.~B\"ungener and M.~Kaufmann,
\newblock Improving the Crossing Lemma by characterizing dense 2-planar and 3-planar graphs,
\newblock \emph{J. Graph Algorithms Appl.} 29(3) (2026), 143--174. \href{https://doi.org/10.7155/jgaa.v29i3.3000}{doi:10.7155/jgaa.v29i3.3000}.

\bibitem{FoxPachSuk}
J.~Fox, J.~Pach and A.~Suk,
\newblock Immersions and Albertson's conjecture,
\newblock in \emph{41st International Symposium on Computational Geometry (SoCG 2025)},
\newblock LIPIcs 332 (2025), 50:1--50:10. \href{https://doi.org/10.4230/LIPIcs.SoCG.2025.50}{doi:10.4230/LIPIcs.SoCG.2025.50}.

\bibitem{Gallai}
T.~Gallai,
\newblock Kritische Graphen II,
\newblock \emph{Publ. Math. Inst. Hungar. Acad. Sci.} 8 (1963), 373--395.

\bibitem{Kleitman}
D.~J. Kleitman,
\newblock The crossing number of $K_{5,n}$,
\newblock \emph{J. Combin. Theory} 9 (1970), 315--323.

\bibitem{LuizRichter}
A.~G. Luiz and R.~B. Richter,
\newblock Remarks on a conjecture of Bar\'at and T\'oth,
\newblock \emph{Electron. J. Combin.} 21 (2014), P1.57.

\bibitem{Shannon}
C.~E. Shannon,
\newblock A theorem on coloring the lines of a network,
\newblock \emph{J. Math. Phys.} 28 (1949), 148--152.

\bibitem{FoxPachSukExpanded}
J.~Fox, J.~Pach and A.~Suk,
\newblock Immersions and Albertson's conjecture,
\newblock preprint, \href{https://arxiv.org/abs/2510.05893v1}{arXiv:2510.05893v1}, 2025.

\bibitem{Cranston}
D.~W. Cranston,
\newblock Progress on Albertson's Conjecture,
\newblock preprint, \href{https://arxiv.org/abs/2512.08020v1}{arXiv:2512.08020v1}, 2025.

\bibitem{CaoMehat}
S.~Cao and S.~S. Mehat,
\newblock Albertson's Conjecture for chromatic numbers at most $29$,
\newblock preprint, \href{https://arxiv.org/abs/2609.04771v1}{arXiv:2609.04771v1}, 2026.

\bibitem{Krivelevich}
M.~Krivelevich,
\newblock On the minimal number of edges in color-critical graphs,
\newblock \emph{Combinatorica} 17 (1997), 401--426.

\bibitem{KellyPostle}
T.~Kelly and L.~Postle,
\newblock On the density of critical graphs with no large cliques,
\newblock \emph{Combinatorica} 43 (2023), 57--89.

\bibitem{FoxTidorZhang}
J.~Fox, J.~Tidor and S.~Zhang,
\newblock Triangle Ramsey numbers of complete graphs,
\newblock \emph{J. Combin. Theory Ser. B} 176 (2026), 268--286.

\bibitem{Garloff}
J.~Garloff,
\newblock The Bernstein algorithm,
\newblock \emph{Interval Computations} 1993, no.~2, 154--168.

\bibitem{VerificationArchive}
A.~Sadhu,
\newblock Verification Code and Certificates for  \emph{A Proof of Albertson's Conjecture},
\newblock Zenodo (2026).\newline
\href{https://doi.org/10.5281/zenodo.22719759}{doi:10.5281/zenodo.22719759}.

\end{thebibliography}
\end{document}